\documentclass[12pt,a4paper,reqno]{amsart}

\usepackage[T1]{fontenc}
\usepackage[latin1]{inputenc}
\usepackage[usenames,dvipsnames]{color}

\usepackage{amsmath,amsfonts,amssymb, amsthm}
\usepackage{mdwlist}

\usepackage{amscd}
\usepackage[all,cmtip]{xy}
\usepackage{fancyhdr}
\usepackage{hyperref}
\usepackage{nameref}

\newtheorem{thm}{Theorem}[section]
\newtheorem*{thm*}{Theorem}
\newtheorem{lm}[thm]{Lemma}
\newtheorem{prop}[thm]{Proposition}
\newtheorem{cor}[thm]{Corollary}
\theoremstyle{definition}\newtheorem{defi}[thm]{Definition}
\theoremstyle{definition}\newtheorem{ex}[thm]{Example}
\theoremstyle{remark}\newtheorem{rem}[thm]{Remark}

\DeclareMathOperator{\he}{ht}

\DeclareMathOperator{\depth}{depth}
\DeclareMathOperator{\spec}{Spec}

\DeclareMathOperator{\Quot}{Quot}

\DeclareMathOperator{\codim}{codim}
\DeclareMathOperator{\Hom}{Hom}

\DeclareMathOperator{\Hilb}{Hilb}

\DeclareMathOperator{\coker}{coker}

\DeclareMathOperator{\id}{id}
\DeclareMathOperator{\pr}{pr}

\DeclareMathOperator{\supp}{Supp}
\DeclareMathOperator{\Ass}{Ass}
\DeclareMathOperator{\Ann}{Ann}

\newcommand{\ra}{\rightarrow}
\newcommand{\surj}{\twoheadrightarrow}
\newcommand{\func}[3]{\ensuremath{#1\mathpunct: #2\ra #3}}

\newcommand{\injfunc}[3]{\ensuremath{#1\mathpunct: #2\hookrightarrow #3}}
\newcommand{\iso}[3]{\ensuremath{#1\mathpunct:
    #2\stackrel{_\sim}{\longrightarrow} #3}}

\newcommand{\CM}{\ensuremath{\mathit{CM}} }
\newcommand{\opn}[1]{\operatorname{#1}}

\newcommand{\calA}{\ensuremath{{\mathcal{A}}}}

\newcommand{\calE}{\ensuremath{{\mathcal{E}}}}
\newcommand{\calF}{\ensuremath{{\mathcal{F}}}}
\newcommand{\calG}{\ensuremath{{\mathcal{G}}}}
\newcommand{\calH}{\ensuremath{{\mathcal{H}}}}
\newcommand{\calI}{\ensuremath{{\mathcal{I}}}}
\newcommand{\calJ}{\ensuremath{{\mathcal{J}}}}
\newcommand{\calK}{\ensuremath{{\mathcal{K}}}}
\newcommand{\calL}{\ensuremath{{\mathcal{L}}}}

\newcommand{\calN}{\ensuremath{{\mathcal{N}}}}
\newcommand{\calO}{\ensuremath{{\mathcal{O}}}}
\newcommand{\calP}{\ensuremath{{\mathcal{P}}}}

\newcommand{\calZ}{\ensuremath{{\mathcal{Z}}}}

\newcommand{\p}{\ensuremath{\mathfrak{p}}}
\newcommand{\q}{\ensuremath{\mathfrak{q}}}
\newcommand{\m}{\ensuremath{\mathfrak{m}}}

\newcommand{\bbN}{\ensuremath{{\mathbb{N}}}}
\newcommand{\bbP}{\ensuremath{{\mathbb{P}}}}
\newcommand{\bbQ}{\ensuremath{{\mathbb{Q}}}}
\newcommand{\bbZ}{\ensuremath{{\mathbb{Z}}}}

\begin{document}
\title{The space of Cohen--Macaulay curves}
\author{Katharina Heinrich} 
\address{KTH Royal Institute of Technology, Institutionen f\"or matematik, 10044 Stockholm, Sweden}
\email{kchal@math.kth.se}

\begin{abstract}
  One can consider the Hilbert scheme as a natural compactification of
  the space of smooth projective curves with fixed Hilbert
  polynomial. Here we consider a different modular compactification,
  namely the functor $\CM$ parameterizing curves together with a
  finite map to $\bbP^n$ that is generically a closed immersion.

  We prove that $\CM$ is an algebraic space by contructing a scheme
  $W$ and a representable, surjective and smooth map $\func \pi W
  CM$. Moreover, we show that $\CM$ satisfies the valuative criterion
  for properness.
\end{abstract}
\maketitle

\thispagestyle{empty}

\section{Introduction}

A classical problem in algebraic geometry is to find and to describe
moduli spaces of different geometric objects. The objects considered
here are embedded projective {\em curves}, that is, one-dimensional
sub\-schemes of a given projective space that do not have embedded or
isolated points. Moreover, we assume that the degree $d$ and the genus
$g$ of the curve are fixed, that is, that the curve has a given
Hilbert polynomial $p(t)=dt+1-g$.

A well-known compactification of the space of all curves in a given
projective space $\bbP^n$ is the {\em Hilbert scheme}. For a fixed
polynomial $p(t)$, the Hilbert functor $\calH ilb^p$ parameterizes
flat families of closed subschemes of $\bbP^n$ with Hilbert polynomial
$p(t)$. Grothendieck defined this functor in \cite{FGA221} and proved
that it is represented by a projective scheme $\Hilb^p$. The Hilbert
scheme compactification of the space of embedded curves is obtained by
pa\-ra\-me\-terizing not only curves but all {\em subschemes} having
Hilbert polynomial $p(t)$. For example, the Hilbert scheme
$\Hilb^{3t+1}$ of twisted cubics in $\bbP^3$ contains not only the
smooth twisted cubic curves but also plane curves with embedded or
isolated points.

The moduli space of curves without embedded or isolated points
that we are interested in is represented by an open subscheme of the
Hilbert scheme, and we seek a compactification that avoids the
degenerated schemes mentioned above.

H\o nsen proposed in \cite{Honsen} the following modular
compactification of the space of curves. Instead of looking at curves
with an embedding into the given projective space, he considered
curves with a finite map to the projective space that is generically a
closed immersion. Concretely, this means that the map is an
isomorphism onto its image away from a finite set of closed
points. For example, the normalization of a plane nodal curve in
$\bbP^3$ gives a finite map to $\bbP^3$ that is an isomorphism onto
the image away from the singularity.

H\o nsen proved that the moduli space $\CM$ of such pairs $(C,i)$,
where $\func i C \bbP^n$ denotes the finite map, is an algebraic space
by verifying the conditions in {\em Artin's criteria} for
representability \mbox{\cite[Theorem 3.4]{Artin}}. Furthermore, he
showed that $\CM$ satisfies the valuative criterion for properness, and
thus $\CM$ is a proper algebraic space.

Studying \cite{Honsen}, we had trouble following several of the
arguments as they appeared to be incomplete or essential conditions
seemed not to be satisfied. Here, we construct a scheme $W$ with a
representable, surjective and smooth cover \mbox{$\func
  \pi{W}\CM$}. Thereby, we re-prove that $\CM$ is an algebraic
space. Parts of the construction are based on the ideas in
\cite{Honsen}. However, we were able to simplify several of the
arguments. Moreover, we give a different, more explicit proof for
properness using ideas as in \cite{Kollar:HH}.

A similar moduli space was constructed by Alexeev and Knutson. They
showed in \cite{AK:branched} that the space of {\em
  branchvarieties} parameterizing reduced curves with finite maps to
$\bbP^n$ is a proper Artin stack. Related moduli spaces can also be
found in \cite{Rydh:thesis}, \cite{Kollar:HH} and
\cite{RT:stable_pairs}.

 \subsubsection*{Notation and conventions}
 All schemes considered here are locally Noetherian. In particular,
 $\mathbf{Sch}$ and $\mathbf{Sch}_S$ denote the categories of locally
 Noetherian schemes and $S$-schemes respectively.

 Let $X$ and $Y$ be $S$-schemes, and let $\func f X Y$ be a morphism
 over $S$. For a base change $\func g T S$, we write $X_T$ for the
 fiber product $X\times_S T$ and $f_T$ for the induced morphism
 $\func{f\times \id_T}{X_T}{Y_T}$ of $T$-schemes. If $T=\spec(A)$ for
 a ring $A$, we write $X_A$ and $f_A$ instead. Moreover, for
 $T=\spec(\kappa(s))$ with $s\in S$, we use the notation $X_s$ and
 $f_s$.

 Let $\calF$ be a quasi-coherent sheaf on a $S$-scheme $X$. For a
 base change $\func g T S$, we denote the pullback $h^*\calF$, where
 $\func h {X_T}X$ is the projection, by $\calF_T$.  In particular, for
 $T=\spec(\kappa(s))$ we write $\calF_s$. This should not be confused
 with the stalk $\calF_x$ of $\calF$ at some point $x\in X$.

 \subsection*{Acknowledgments} This project was carried out as part of
 my PhD thesis. I thank my advisor Roy Skjelnes for his guidance and
 support. Moreover, I am grateful to David Rydh. The construction of
 the refinement of the cover in Subsection~\ref{sec:refinement} is
 very much based on his ideas.

\section{Families of Cohen--Macaulay curves}

In this section, we give a definition of the space $\CM$ of
Cohen--Macaulay curves and investigate some simple examples. Then we
show that a Cohen--Macaulay curve does not have nontrivial
automorphisms. Finally, this result is used to prove that $\CM$ is a
sheaf in the \'etale topology.

\subsection{Definition of the Cohen--Macaulay functor}\label{sec:defCM}
\begin{defi}\label{def:CMfunctor}
  Let $p(t)=at+b\in\bbZ[t]$ be a polynomial of degree $1$. For a scheme
  $S$, let $\CM(S):=\CM_{\bbP^n}^{p(t)}(S)$ be the set of isomorphism
  classes of pairs $(C,i)$ of a flat $S$-scheme $C$ and a finite
  $S$-morphism $\func i C{\bbP^n_S}$ such that for every $s\in S$ the
  following properties hold.
  \begin{enumerate}
  \item The fiber $C_s$ is a Cohen--Macaulay scheme, that is, all local
    rings are Cohen--Macaulay rings, and it has pure dimension $1$.
  \item The induced morphism $\func{i_s}{C_s}{\bbP^n_{\kappa(s)}}$ is
    an isomorphism onto its image away from a finite set of closed
    points.
  \item\label{def:CM_hilb} The coherent sheaf $(i_*\calO_C)_s$ on
    $\bbP^n_{\kappa(s)}$ has Hilbert polynomial $p(t)$.
  \end{enumerate}
  Two such pairs $(C_1,i_1)$ and $(C_2,i_2)$ in $\CM(S)$ are
  considered equal if there exists an $S$-isomorphism
  $\iso\alpha{C_1}{C_2}$ such that the diagram
 $$\xymatrix{ & \bbP^n_S & \\
   C_1\ar[rr]^-\alpha\ar[ur]^{i_1} & & C_2\ar[ul]_{i_2}}$$ commutes.
\end{defi} 

\begin{rem}
  Note that $(i_*\calO_C)_s=(i_s)_*\calO_{C_s}$ since the morphism $i$
  is affine.  Moreover, we have $((i_s)_*\calO_{C_s})(d) =
  (i_s)_*(i_s^*\calO_{\bbP^n_{\kappa(s)}}(d))$ for every $d\in\bbZ$ by
  the projection formula.  Hence it follows for all $r\geq 0$ that
  $H^r(\bbP^n_{\kappa(s)},(i_*\calO_C)_s(d)) =
  H^r(C_s,i_s^*\calO_{\bbP^n_{\kappa(s)}}(d))$. In particular,
  property \ref{def:CM_hilb} in the definition is equivalent to
  requiring that the structure sheaf $\calO_{C_s}$ has Hilbert
  polynomial $p(t)$ with respect to the ample invertible sheaf
  $i_s^*\calO_{\bbP^n_{\kappa(s)}}(1)$.
\end{rem}

\subsubsection{\texorpdfstring{Functoriality of $\CM$}{Functoriality
    of CM}} We show in Proposition~\ref{prop:functor} that the
assignment $\CM$ indeed defines a functor.

\begin{lm}\label{lm:iso_alt}
  Let $\func f X Y$ be a finite morphism of locally Noetherian
  $S$-schemes, and let $Z:=\supp(\coker(\calO_Y\to f_*\calO_X))$. For
  $s\in S$, the induced map $\func{f_s}{X_s}{Y_s}$ is an isomorphism
  onto its image away from a finite set of closed points if and only
  if $\dim(Z_s)=0$.
\end{lm}
\begin{proof}
  Let $s\in S$. The finite morphism $f_s$ is an isomorphism onto its
  image away from the closed locus $\supp(\coker(f_s^\#))$ in $Y_s$,
  where $\func{f_s^\#}{\calO_{Y_s}}{(f_s)_*\calO_{X_s}}$ is the
  natural map. Now the statement follows as
  $\supp(\coker(f_s^\#))=Z_s$ as closed subsets of $Y_s$.
\end{proof}

\begin{lm}\label{lm:func_field_change}
  Let $X$ be a scheme locally of finite type over a field $k$, and
  let $k\subseteq K$ be a field extension.
  \begin{enumerate}
  \item\label{lm:func_field_change_CM} The scheme $X$ is
    Cohen--Macaulay if and only if $X_K$ is Cohen--Macaulay.
  \item\label{lm:func_field_change_pure} The scheme $X$ has
    pure dimension $n$ if and only if $X_K$ has pure dimension~$n$.
  \item\label{lm:func_field_change_isom} We have $\dim(X_K)=\dim(X)$.
  \end{enumerate}
\end{lm}
\begin{proof}
  \begin{enumerate}
  \item Since that property is local on $X$, we can reduce to the
    affine case that is shown in \cite[Theorem 2.1.10]{Bruns-Herzog}.
  \item By \cite[Corollaire (4.2.8)]{EGAIV2}, the set of dimensions of
    the irreducible components of $X$ equals the set of dimensions of
    the irreducible components of $X_K$, and the statement follows.
  \item This is \cite[Corollaire (4.1.4)]{EGAIV2}.\qedhere
  \end{enumerate}
\end{proof}
\begin{prop}\label{prop:functor}
  The assignment $S\mapsto \CM(S)$ defines a contravariant functor
  $\func{\CM}{\mathbf{Sch}^\circ}{\mathbf{Sets}}$.
\end{prop}
\begin{proof}
  Let $\func f{T}{S}$ be a morphism, and let $(C,i)\in
  \CM(S)$. We claim that $(C_T,i_T)\in \CM(T)$. First we observe that
  being flat and being finite are stable under base change. Moreover,
  the fiber of $C_T$ over some $t\in T$ is the base change of the
  fiber of $C$ over the point $f(t)$ by the field extension
  $\kappa(f(t))\hookrightarrow \kappa(t)$. As the fiber $C_{f(t)}$ is
  locally of finite type over $\spec(\kappa(t))$, it follows from
  Lemma~\ref{lm:func_field_change} and Lemma~\ref{lm:iso_alt}
  that $(C_T,i_T)\in \CM(T)$.
\end{proof}

\subsubsection{Cohen--Macaulay rings of dimension $1$} In the case of
one-di\-men\-sional rings, the Cohen--Macaulay property is particularly
simple.
\begin{lm}\label{lm:CM_ring}
  Let $A$ be a Noetherian ring of dimension $1$. The following are
  equivalent.
  \begin{enumerate}
  \item The ring $A$ is Cohen--Macaulay.\label{lm:CM_ring1}
  \item All associated prime ideals of $A$ are minimal, that is, $A$
    does not have embedded prime ideals.\label{lm:CM_ring2}
  \item Every zero divisor of $A$ is contained in a minimal prime
    ideal of $A$.\label{lm:CM_ring3}
  \end{enumerate}
\end{lm}
\begin{proof}
  According to the {\em Unmixedness Theorem} \cite[Theorem
  2.1.6]{Bruns-Herzog}, a ring is Cohen--Macaulay if and only if every
  ideal $I$ generated by $\he(I)$ elements is unmixed, that is, all
  associated prime ideals of $I$ are minimal over $I$.

  As every ideal of maximal height is unmixed and $\dim(A)=1$, it
  follows that $A$ is Cohen--Macaulay if and only if the ideal $(0)$ is
  unmixed. But the latter means that $A$ does not have any embedded
  prime ideals. This shows the equivalence of
  assertions~\ref{lm:CM_ring1} and \ref{lm:CM_ring2}.

  That assertion~\ref{lm:CM_ring2} implies assertion~\ref{lm:CM_ring3}
  follows immediately from the fact that the set
  $\operatorname{Zdv}(A)$ of zero divisors of $A$ equals the union of
  the associated prime ideals of $A$.  Suppose conversely that every
  zero divisor is contained in a minimal prime ideal. Then we have the
  inclusion
  $$\bigcup_{\text{ass.\ primes } \p}\p=\operatorname{Zdv}(A)\subseteq
  \bigcup_{\text{minimal primes }\p}\p.$$ By {\em Prime
    avoidance}~\cite[Lemma 3.3]{Eisenbud}, it follows that every
  associated prime ideal is minimal, and we are done.
\end{proof}

\subsubsection{Families over a local Artin ring}
\begin{lm}\label{lm:equidim}
  Let $X$ be a scheme of finite type over $\spec(k)$, where $k$ is a
  field. Then the following properties are equivalent.
  \begin{enumerate}
  \item The scheme $X$ has pure dimension $n$.
  \item Every nonempty open subscheme $U$ of $X$ has pure dimension~$n$.
  \end{enumerate}
\end{lm}
\begin{proof}
  First we can without loss of generality assume that $X$ and hence
  also every open subscheme of $X$ is reduced. Moreover, we observe
  that $\dim(V)=\dim(Y)$ for every nonempty open subscheme $V$ of an
  integral scheme $Y$ of finite type over $k$, see
  \cite[(4.1.1.3)]{EGAIV2}.

  Suppose that $X$ has pure dimension $n$, and let $\emptyset\neq
  U\subseteq X$ be an open subscheme. The irreducible components of
  $U$ are of the form $Y\cap U$ where $Y$ is an irreducible component
  of $X$. By the remark above, it follows that $\dim(Y\cap
  U)=\dim(Y)=n$.

  The opposite implication follows directly by setting $U=X$.
\end{proof}

\begin{lm}\label{lm:equidim_Artin}
  Let $X$ be a scheme of finite type over $\spec(R)$, where $R$ is a
  local Artin ring.  Suppose that the closed fiber $X_0$ has pure
  dimension $n$. Then $X$ has pure dimension $n$. Moreover, every
  nonempty open subscheme $U$ of $X$ has pure dimension $n$.
\end{lm}
\begin{proof}
  Observe first that every nonempty open subscheme $V$ of $X_0$ has
  pure dimension $n$ by Lemma~\ref{lm:equidim}. Now the statement
  follows immediately since $X$ and $X_0$ have the same underlying
  topological space.
\end{proof}
\begin{prop}\label{prop:fiber_CM}
  Let $\func f X S$ be a flat morphism of locally Noetherian schemes.
  \begin{enumerate}
  \item \label{prop:fiber_CM1} Suppose that $S$ and all fibers $X_s$,
    where $s\in S$, are Cohen--Macaulay schemes.  Then $X$ is
    Cohen--Macaulay.
  \item \label{prop:fiber_CM2} Let $S=\spec(R)$ for a local Artin ring
    $R$, and suppose that the closed fiber $X_0$ is Cohen--Macaulay and
    of pure dimension $n$. Then $X$ is Cohen--Macaulay and of pure
    dimension $n$.
  \end{enumerate}
\end{prop}
\begin{proof}
  Since the statement is local on both $X$ and $S$, we can without
  loss of generality assume that $X=\spec(A)$ and $S=\spec(R)$ are
  affine, where $A$ is a flat $R$-algebra and $R$ is
  Cohen--Macaulay. Then assertion~\ref{prop:fiber_CM1} is the statement
  of \cite[Proposition 2.1.16(b)]{Bruns-Herzog}.

  For assertion~\ref{prop:fiber_CM2}, we observe first that $S$ is
  Cohen--Macaulay. In particular, the scheme $X$ is Cohen--Macaulay by
  assertion \ref{prop:fiber_CM1}. Moreover, it follows from
  Lemma~\ref{lm:equidim_Artin} that $X$ has pure dimension $n$.
\end{proof}
\begin{prop}\label{prop:Artin_base}
  Let $(C,i)\in \CM(\spec(R))$ for a local Artin ring $R$. Then $C$ is
  Cohen--Macaulay and of pure dimension $1$.  Moreover, the morphism
  $i$ is an isomorphism onto its image away from finitely many closed
  points.
\end{prop}
\begin{proof}
  We have seen in Proposition~\ref{prop:fiber_CM} that the scheme $C$
  is Cohen--Macaulay and of pure dimension $1$.

  Let $Y=\supp(\coker(\calO_{\bbP^n_R}\to i_*\calO_C))$. Then the
  closed fiber $Y_0$ is zero-dimensional by
  Lemma~\ref{lm:iso_alt}. Now the statement follows as $Y$ and
  $Y_0$ have the same underlying topological space and, in particular,
  the same dimension.
\end{proof}

\subsubsection{Scheme-theoretic image} Next we study the
scheme-theoretic image $i(C)$ of a point $(C,i)\in \CM(\spec(k))$. In
particular, we show that it is Cohen--Macaulay and of pure
dimension $1$.
\begin{lm}\label{lm:image_affine}
  Let $\injfunc j B A$ be a finite injective homomorphism of
  Noetherian rings, and suppose that $A$ is Cohen--Macaulay and of pure
  dimension $1$. Then $B$ is Cohen--Macaulay of pure dimension $1$.
\end{lm}
\begin{proof}
  By {\em Cohen--Seidenberg Theorem} \cite[(5.E) Theorem
  5]{Mats:algebra}, we have $\dim(B)=\dim(A)$ and every minimal
  prime $\p$ of $B$ is contained in a minimal prime $\q$ of $A$ with
  $\dim(B/\p)=\dim(A/\q)$. It follows that $B$ has pure dimension $1$.

  By Lemma~\ref{lm:CM_ring}, a one-dimensional ring is Cohen--Macaulay
  if and only if every zero divisor is contained in a minimal prime
  ideal. So let $b$ be a zero divisor of $B$. Then its image $j(b)$ is
  a zero divisor of $A$, and hence it is contained in a minimal prime
  ideal $\q$ of $A$. The restriction $\p:=j^{-1}(\q)$ is a prime ideal
  of $B$ containing $b$ with $\dim(B/\p)=\dim(A/\q)$. It follows that
  $\dim(B/\p)=1$, and $\p$ is minimal.
\end{proof}
\begin{prop}\label{prop:image}
  Let $\func f X Y$ be a finite morphism of schemes of finite type
  over a field. Suppose that $X$ is Cohen--Macaulay and of pure
  dimension $1$. Then the scheme-theoretic image $Z$ of $X$ in $Y$ is
  Cohen--Macaulay and of pure dimension $1$.
\end{prop}
\begin{proof}
  By definition and Lemma~\ref{lm:equidim} respectively, both
  conditions are local, and we reduce to the case that $X=\spec(A)$
  and $Y=\spec(C)$ are affine.  Then $Z=\spec(B)$ with $B=C/\ker(C\ra
  A)$. By Lemma~\ref{lm:image_affine}, the ring $B$ is Cohen--Macaulay
  of pure dimension $1$.
\end{proof}
\begin{cor}\label{cor:scheme-theor-image_CM}
  Let $(C,i)\in \CM(\spec(k))$ for a field $k$. Then the
  scheme-theoretic image of $C$ in $\bbP^n_k$ is Cohen--Macaulay and of
  pure dimension $1$.
\end{cor}
\begin{proof}
  This follows directly from Proposition~\ref{prop:image}.
\end{proof}
Moreover, we can give bounds for the degree and the genus of the
projective curve $i(C)$.

\begin{prop}\label{prop:image_Hilb_poly}
  Let $k$ be a field, and let $(C,i) \in
  \CM_{\bbP^n}^{at+b}(\spec(k))$. Then the scheme-theoretic image
  $i(C)$ of $C$ in $\bbP^n_k$ is a curve of degree $a$ and arithmetic
  genus $g$ that is bounded by \[ 1-b\leq g \leq \frac 1 2
  (a-1)(a-2).\] In particular, there are only finitely many
  possibilities for the Hilbert polynomial $p_{i(C)}(t)$ of the curve
  $i(C)$.
\end{prop}
\begin{proof}
  The finite morphism $i$ factors through the scheme-theoretic
  image as $$\xymatrix{C\ar[rr]^i\ar[dr]^-{h} && \bbP^n_k, \\
    &i(C)\ar@{^(->}[ur]^-j&}$$ where $h$ is finite and an isomorphism
  away from finitely many closed points, and $j$ is a closed
  immersion.  The one-dimensional scheme $i(C)$ has Hilbert polynomial
  of degree $1$, say $p_{i(C)}(t)=ct+d$.

  The cokernel $\calK$ of the injective map $j_*\calO_{i(C)}\ra
  i_*\calO_C$ is supported on the finitely many closed points that
  form the non-isomorphism locus of $h$.  Hence its Hilbert polynomial
  $p_{\calK}(t)$ is constant, $p_{\calK}(t) =h^0(\calK)=:l$ for
  $l\in\bbN$.  Consider the short exact sequence
  \begin{equation*}
    \xymatrix{0\ar[r] & j_*\calO_{i(C)}\ar[r] & i_*\calO_C\ar[r] &
      \calK\ar[r] & 0} \end{equation*}
  of $\calO_{\bbP^n_k}$-modules. The additivity of the Hilbert
  polynomial implies then that  
  \begin{equation}
    p_{i_*{\calO_C}}(t)=p_{i(C)}(t)+p_\calK(t),\label{eq:Hilb_poly}
  \end{equation}
  that is, $at+b=ct+d+l$. In particular, we get that $c=a$ and $d\leq
  b=d+l$.  The relation $d=1-g$ gives the lower bound $1-b\leq
  g$. Moreover, the arithmetic genus of a curve of degree $a$ is
  bounded from above by $\frac 1 2 (a-1)(a-2)$ by
  \cite[Theorem~3.1]{Hartshorne:genus}. Note that the proof there
  works even for curves in $\bbP^n_k$ with $n\geq 3$.
\end{proof}

\subsection{First examples} The results of
Proposition~\ref{prop:image_Hilb_poly} allow us to understand some
simple examples.
\begin{thm} Let $p(t)=at+b\in \bbZ[t]$ be a polynomial.
  \begin{enumerate}
  \item \label{item:ex_empty} Suppose that $b<1-\frac 1 2
    (a-1)(a-2)$. Then $\CM_{\bbP^n}^{at+b}=\emptyset$.
  \item If $b=1-\frac 1 2 (a-1)(a-2)$, then
    $CM_{\bbP^n}^{at+b}=\Hilb_{\bbP^n}^{at+b}$ is the Hilbert scheme
    of plane curves of degree $a$ in $\bbP^n$.
  \end{enumerate}
\end{thm}
\begin{proof}
  Let $k$ be a field, and let $(C,i)\in \CM(\spec(k))$. Then by
  Proposition~\ref{prop:image_Hilb_poly}, the scheme-theoretic image
  $i(C)$ of $C$ in $\bbP^n_k$ is a curve of degree $a$ and arithmetic
  genus $g$ such that $1-b\leq g\leq 1 - \frac 1 2 (a-1)(a-2)$. This
  show directly that no such curves exists if $b<1-\frac 1
  2(a-1)(a-2)$, and hence assertion~\ref{item:ex_empty}.

  If $b=1-\frac 1 2 (a-1)(a-2)$, it follows that $g=1-b$. Observe that
  the image $i(C)$ has Hilbert polynomial
  $p(t)=p_{i_*\calO_C}(t)$ in this case. Then
  Equation~(\ref{eq:Hilb_poly}) above implies that
  \mbox{$\coker(\calO_{\bbP^n_k}\ra i_*\calO_C)=0$}, that is, the
  curves $C$ and $i(C)$ are isomorphic, and the map $i$ is a closed
  immersion. Since being a closed immersion can be checked on the
  fibers, see \cite[Proposition~(4.6.7)]{EGAIII1}, it follows that for
  every scheme $S$ and every family $(C,i)\in \CM(S)$ the map $i$ is a
  closed immersion. Conversely, every closed subscheme of $\bbP^n_k$
  with Hilbert polynomial $p(t)=at+1-\frac 1 2 (a-1)(a-2)$ is a plane
  curve without embedded or isolated points by
  \cite[Theorem~3.1]{Hartshorne:genus}.
\end{proof}

\subsection{\texorpdfstring{Automorphisms of $\CM$ curves}{Automorphisms of CM curves}}
 
Next we show that for every scheme $S$, and every $(C,i)\in \CM(S)$,
the scheme $C$ does not have any nontrivial automorphism over
$\bbP^n_S$.
\begin{lm}\label{lm:surj_loc}
  Let $\func \varphi B A$ be a finite homomorphism of Noetherian
  rings, where $A$ is Cohen--Macaulay and of pure dimension $1$.
  Suppose that the localized map $\func{\varphi_\p}{B_\p}{A_\p}$ is
  surjective for all prime ideals $\p$ of $B$ apart from finitely many
  maximal ideals $\m_1,\dotsc,\m_s$. Then there exists an element
  $b\in B$ that is not a zero divisor of $A$ such that the map
  $\func{\varphi_b}{B_b}{A_b}$ is surjective.
\end{lm}
\begin{proof}
  An element $b$ with the required properties has to be contained in
  all exceptions $\m_1,\dotsc,\m_s$. Moreover, in order not to be a
  zero divisor of $A$, the element $b$ may not be contained in any of
  the associated prime ideals of $A$ as a $B$-module. We claim that
  such an element exists, that is, $\bigcap_{i=1}^s\m_i\not\subseteq
  \bigcup_{j=1}^t\p_j$, where we set
  \mbox{$\Ass_B(A)=\{\p_1,\dotsc,\p_t\}$}. Suppose namely the
  opposite, that is, $\bigcap_{i=1}^s\m_i \subseteq
  \bigcup_{j=1}^t\p_j$. Prime avoidance implies then that
  $\m_i\subseteq \p_j$ for some $i$ and $j$.  In particular, it
  follows that $\p_j$ is a maximal ideal of $B$. Note that by
  \cite[(9.A)~Proposition]{Mats:algebra}, we have
  $\Ass_B(A)=\{\varphi^{-1}(\q)\mid \q\in \Ass_A(A)\}$, and hence
  $\p_j=\varphi^{-1}(\q)$ for some $\q\in \Ass_A(A)$. The induced map
  $B/{\p_j}\hookrightarrow A/\q$ is a finite and hence integral
  extension of a field, and therefore $A/\q$ is a field. This
  contradicts the fact that $\q$ is an associated prime ideal of the
  Cohen--Macaulay ring $A$ of pure dimension $1$.
\end{proof}
\begin{lm}\label{lm:prod_ass}
  Let $R$ be a Noetherian ring, and let $M$ be a finitely generated
  $R$-module. Then the natural map $M\to \prod_{\p\in\Ass(M)}M_\p$ is
  injective.
\end{lm}
\begin{proof}
  Suppose that the element $m\in M$ is mapped to 0. Then for every
  associated prime ideal $\p\in \Ass(M)$ there exists $s\notin \p$
  such that $sm=0$. In particular, $\Ann(m)\not\subseteq \p$ for all
  associated prime ideals $\p$ of $M$. It follows that $\Ann(m)=R$,
  and hence $m=0$.
\end{proof}
\begin{lm}\label{lm:subring_prod_Artin}
  Let $R$ be a Noetherian ring. Then $R$ has an embedding into a
  finite product of local Artin rings.
\end{lm}
\begin{proof}
  Let $(0)=\bigcap_{i=j}^l\q_j$ be a {\em primary decomposition} of
  the zero ideal of $R$, see \cite[Theorem IV.2.2.1]{Bourbaki:CA}. The
  natural map $R\to \prod_{j=1}^l R/\q_j$ is then injective. The ideal
  $\q_j$ is $\p_j$-primary for $\p_j:=\sqrt{\q_j}$, and hence we have
  that $\Ass(R/\q_j)=\{\p_j/\q_j\}$. In particular, it follows by
  Lemma~\ref{lm:prod_ass} that the localization map \mbox{$R/\q_j\to
    R_j:=(R/\q_j)_{\p_j/\q_j}$} is injective. Since $\p_j$ is a
  minimal prime over ideal of $\q_j$, the ring $R_j$ is a local Artin
  ring, and the composition $R\hookrightarrow \prod_{j=1}^l R/\q_j
  \hookrightarrow \prod_{j=1}^l R_j$ has the required properties.
\end{proof}
\begin{thm}\label{thm:auto}
  Let $(C,i)\in \CM(S)$ for a locally Noetherian scheme $S$.  Let
  $\alpha$ be an $S$-automorphism of $C$ that is compatible with the
  finite morphism $i$, that is, $i\circ\alpha =i$. Then $\alpha$ is
  the identity morphism on $C$.
\end{thm}
\begin{proof}
  Suppose first that $S=\spec(R)$ for a local Artin ring $R$. By
  Proposition~\ref{prop:Artin_base}, the scheme $C$ is Cohen--Macaulay
  and of pure dimension $1$, and $i$ is an isomorphism onto its image
  away from finitely many closed points. Let $V=\spec(B)\subset
  \bbP^n_R$ be an open subset. Then $i^{-1}(V)=\spec(A)$ is affine, and
  $\alpha$ induces an endomorphism of $\spec(A)$. In particular, we
  can reduce to the affine situation, and we have to show that every
  $B$-algebra endomorphism $\alpha$ of $A$ is the identity. Note that
  the ring $A$ is Cohen--Macaulay and of pure dimension $1$, and the
  structure map $\func \varphi B A$ satisfies the conditions of
  Lemma~\ref{lm:surj_loc}. Hence there exists $b\in B$ such that
  $B_b\ra A_b$ is surjective and $b$ is not a zero divisor of $A$. It
  follows that the induced $B_b$-algebra endomorphism $\alpha_b$ is
  the identity. Moreover, we have an inclusion $A\subseteq A_b$. Now
  the statement follows as $\alpha$ is the restriction of $\alpha_b$
  to $A$.

  Secondly, we consider the case that $S=\spec(R)$ for a Noetherian
  ring $R$. As in the first case it suffices to show that for any open
  affine subset $V=\spec(B)\subset \bbP^n_R$ and $i^{-1}(V)=\spec(A)$,
  every $B$-algebra endomorphism $\alpha$ of $A$ is the identity. Let
  $R\hookrightarrow \prod_{j=1}^l R_j$ be an embedding into a finite
  product of local Artin rings $R_j$, see
  Lemma~\ref{lm:subring_prod_Artin}. Since $A$ is flat over $R$, we
  get an injection $A\hookrightarrow A\otimes_R \prod_{j=1}^l R_j\cong
  \prod_{j=1}^l (A\otimes_R R_j)$. Moreover, $\alpha$ is the
  restriction of the endomorphism $\prod_j (\alpha\otimes\id_{R_j})$
  of $\prod_j (A\otimes_R R_j)$ to $A$. But $\alpha\otimes\id_{R_j}$
  is a $B\otimes_R R_j$-algebra endomorphism of $A\otimes_R R_j$, and
  hence, as in the first case, it is the identity. It follows that
  $\alpha=\id$.

  For a general locally Noetherian scheme $S$, we observe that the
  statement is local on $S$, and we can directly reduce to the affine
  case.
\end{proof}

\begin{cor}\label{cor:auto}
  Let $S$ be a scheme, and let $(C_1,i_1),(C_2,i_2)\in \CM(S)$.  There
  exists at most one isomorphism $\func \beta{C_1}{C_2}$ such that
  $i_2\circ \beta=i_1$.
\end{cor}
\begin{proof}
  Let $\beta_1,\beta_2$ be two such isomorphisms. Then
  $\alpha:=\beta_2^{-1}\circ \beta_1$ is an automorphism of $C_1$ such
  that $i_1\circ \alpha=i_1$, and hence $\beta_1=\beta_2$ by
  Theorem~\ref{thm:auto}.
\end{proof}

\begin{rem}
  Note that Theorem~\ref{thm:auto} does not hold if we drop the
  requirement in Definition~\ref{def:CMfunctor} of $\CM$ that the
  curves $C_s$ have to be of pure dimension $1$. Consider namely the
  $k$-scheme $C$ that is the disjoint union of a line and two isolated
  points $x_1$ and $x_2$, and a finite morphism $\func i C \bbP^n_k$
  that embeds the line and maps the points $x_1$ and $x_2$ to the same
  image. Then the morphism $\func \alpha C C$ that keeps the line
  fixed and exchanges the points $x_1$ and $x_2$ is a nontrivial
  automorphism such that $i \circ \alpha =i$.
\end{rem}

\subsection{Sheaf in the \'etale topology} The nonexistence of
nontrivial automorphisms is an important tool for showing that $\CM$
is a sheaf in the \'etale topology.
\begin{prop}[{\cite[Proposition II.1.5]{Milne:EC}}]\label{prop:sheaf_criterion}
  Let $\func{F}{\mathbf{Sch}^\circ}{\mathbf{Sets}}$ be a
  functor. Then $F$ is a sheaf in the \'etale topology if and only if the
  following properties hold.
  \begin{enumerate}
  \item $F$ is a sheaf in the Zariski topology.
  \item For any surjective \'etale morphism $V\ra U$ of affine schemes
    the diagram $$F(U)\ra F(V) \rightrightarrows F(V\times_U V)$$ is
    exact.
  \end{enumerate}
\end{prop}
\begin{prop}\label{prop:sheaf_Zariski}
  The functor $\CM$ is a sheaf in the Zariski topology.
\end{prop}
\begin{proof}
  Let $S$ be a scheme, and let $\{S_a\}$ be an open cover of $S$. We
  have to show that the sequence $$\CM(S)\to \prod_a
  \CM(S_a)\rightrightarrows \prod_{a,b} \CM(S_{ab}),$$ where we write
  $S_{ab}$ for the intersection $S_a\cap S_b$, is exact.

  Let first $(C,i)$ and $(D,j)$ be two elements of $\CM(S)$
  restricting to the same element of $\CM(S_a)$ for every $a$. Then
  for every $a$ there exists an isomorphism
  $\iso{\alpha_a}{C_{S_a}}{D_{S_a}}$ that commutes with $i$ and $j$.
  Restriction of $\alpha_a$ and $\alpha_b$ to the intersection
  $C_{S_{ab}}$ gives an automorphism $\alpha_b^{-1}\circ \alpha_a$ of
  $C_{S_{ab}}$ that commutes with the finite morphism $i$. Since
  $\alpha_b^{-1}\circ\alpha_a$ is the identity morphism by
  Theorem~\ref{thm:auto}, we see that the morphisms $\alpha_a$
  coincide on the overlaps and hence glue to an isomorphism $\iso
  \alpha C D$ with $j\circ\alpha =i$. This shows that $(C,i)=(D,j)$ as
  elements of $\CM(S)$, and the first map is injective.

  Now let $\{(C_a,i_a)\in \CM(S_a)\}$ be a family that has the same
  image in $\prod_{a,b}\CM(S_{ab})$ under both restriction maps. That
  means that there exist isomorphisms
  $\iso{\alpha_{ab}}{(C_a)_{S_{ab}}}{(C_b)_{S_{ab}}}$ that commute
  with the restrictions of $i_a$ and $i_b$. By Theorem~\ref{thm:auto}
  and Corollary~\ref{cor:auto}, we have equalities
  $\alpha_{ba}=\alpha_{ab}^{-1}$ and
  $\alpha_{ac}=\alpha_{bc}\circ\alpha_{ab}$. Hence the $\{C_a\}$ glue
  to a scheme $C$ over $S$ with a morphism $\func{i}{C}{\bbP^n_S}$ such
  that $i\vert_{C_a}=i_a$ for every $a$. As the defining properties of
  $\CM$ are local on $S$, we see that \mbox{$(C,i)\in \CM(S)$}.
\end{proof}

Let $\func p {S'} S$ be a morphism of schemes. We have projection
morphisms $\func{p_i}{S'\times_S S'}{S'}$ for $i=1,2$ and
$\func{p_{ij}}{S'\times_S S'\times_S S'}{S'\times_S S'}$ for $1\leq
i<j\leq 3$. A {\em descent datum} is a scheme $X'$ over $S'$ together
with an $(S'\times_S S')$-isomorphism $\iso \varphi{p_1^*X'}{p_2^*X'}$
such that $p_{23}^*\varphi\circ p_{12}^*\varphi=p_{13}^*\varphi$. In
particular, for a $S$-scheme $X$ the pullback $p^*X=X_{S'}$
together with the natural isomorphism $\varphi_X\mathpunct :
p_1^*p^*X\cong(p\circ p_1)^*X=(p\circ p_2)^*X\cong p_2^*p^*X$ defines
the so-called canonical descent datum.

Let $\mathbf{Sch}_{S'/S}$ be the category where the objects are
descent data, and morphisms between descent data $(X_1',\varphi_1)$ and
$(X_2',\varphi_2)$ are $S'$-morphisms $\func \tau{X_1'}{X_2'}$ such
that $\varphi_2\circ p_1^*\tau=p_2^*\tau\circ \varphi_1$. Then there
is a functor $\func{G}{\mathbf{Sch}_S}{\mathbf{Sch}_{S'/S}}$ that maps
a $S$-scheme $X$ to the canonical descent datum $(p^*X, \varphi_X)$.
\begin{thm}[{\cite[Theorem
    14.70]{GortzWed:AlgGeo}}] \label{thm:descent_functor} Let $\func
  p{S'}S$ be a faithfully flat and quasi-compact morphism of schemes.
  \begin{enumerate}
  \item The functor $\func G{\mathbf{Sch}_S}\mathbf{Sch}_{S'/S},\
    X\mapsto (p^*X,\varphi_X)$ is fully faithful.
  \item Let $(X',\varphi)$ be a descent datum, and suppose that $X'$ is
    affine over $S'$. Then there exists a scheme $X$ that is affine
    over $S$ and such
    that $(X',\varphi)\cong(p^*X,\varphi_X)$.
  \end{enumerate}
\end{thm}

\begin{prop}\label{prop:sheaf_etale_aff}
  Let $V\to U$ be a faithfully flat morphism of affine schemes $U$ and
  $V$. Then the sequence $$\CM(U)\to \CM(V)\rightrightarrows
  \CM(V\times_U V)$$ is exact.
\end{prop}
\begin{proof}
  Being faithfully flat and quasi-compact is stable under base change.
  Hence also the induced map $\bbP^n_V\ra \bbP^n_U$ is faithfully flat
  and quasi-compact. We will use results on descent as in
  Theorem~\ref{thm:descent_functor} with $S'=\bbP^n_V$ and
  $S=\bbP^n_U$.

  Let $(C,i)$ and $(D,j)$ be two elements of $\CM(U)$ having the same
  image in $\CM(V)$. Then there exists an isomorphism
  $\iso\alpha{C_V}{D_V}$ with $i_V=j_V\circ \alpha$. By
  Theorem~\ref{thm:auto}, we have that $\varphi_{D_V}\circ
  p_1^*\alpha=p_2^*\alpha\circ \varphi_{C_V}$, and hence $\alpha$ is an
  isomorphism between the descent data $(C_V,\varphi_{C_V})$ and
  $(D_V,\varphi_{D_V})$ in $\mathbf{Sch}_{\bbP^n_V/\bbP^n_U}$. Since
  the functor $\func G {\mathbf{Sch}_{\bbP^n_U}}{\mathbf{
      Sch}_{\bbP^n_V/\bbP^n_U}}$ is fully faithful by
  Theorem~\ref{thm:descent_functor}, it follows that $C$ and $D$ are
  isomorphic as schemes over $\bbP^n_U$, and hence $(C,i)=(D,j)$ in
  $\CM(U)$.

  Now suppose that an element $(C_V,i_V)\in \CM(V)$ has the same
  image in $\CM(V\times_UV)$ under both projections. Then we have an
  isomorphism $\iso{\alpha}{p_1^*C_V}{p_2^*C_V}$ of schemes over
  $\bbP^n_{V\times_U V}$. Again by Theorem~\ref{thm:auto}, $\alpha$
  satisfies the cocycle condition so we get a descent datum
  $(C_V,\alpha)$. Since $C_V$ is affine over $S'=\bbP^n_V$, there exists by
  Theorem~\ref{thm:descent_functor} a scheme $C$ with an
  affine morphism $\func i C \bbP^n_U$ such that $C\times_UV\cong C_V$
  as schemes over $\bbP^n_V$. We claim that $(C,i)\in \CM(U)$. As the
  properties finite and flat are stable under faithfully flat descent
  by \cite[Proposition~1.2.36]{FGAexplained}, the morphism $i$ is
  finite and $C$ is flat over $U$. Note further that the map $V\ra U$
  is surjective. Hence every fiber of $C$ over $U$ gives a fiber of
  $C_V$ over $V$ after a suitable change of base field. Then it
  follows from Lemma~\ref{lm:func_field_change} that $(C,i)$ indeed is
  an element of $\CM(U)$.
\end{proof}
\begin{thm}\label{thm:sheaf}
  The functor $\CM$ is a sheaf in the \'etale topology.
\end{thm}
\begin{proof}
  We have seen in Proposition~\ref{prop:sheaf_Zariski} that $\CM$ is a
  sheaf in the Zariski topology. Moreover, since a surjective and \'etale
  morphism is faithfully flat, the sheaf property is satisfied for
  affine \'etale coverings by Proposition~\ref{prop:sheaf_etale_aff}.
  Now the statement follows from
  Proposition~\ref{prop:sheaf_criterion}.
\end{proof}

\section{Representability of the diagonal}

In this section, we show that the diagonal map $\CM\to \CM\times \CM$
is representable. In particular, this implies that every map $X\to
\CM$, where $X$ is a scheme, is representable.

\begin{prop}\label{prop:mor_isom_repres}.
  Let $S$ be a locally Noetherian scheme, and let $X$ and $Y$ be
  schemes that are projective over $S$. Suppose further that $X$ is flat over
  $S$. The functors $({\mathbf{Sch}}_S)^\circ\to {\mathbf{Sets}}$
  defined by
  \begin{align*}
    \calH om_S(X,Y)(T)& :=\Hom_T(X_T,Y_T),\\ \calI som_S(X,Y)(T)&:=
    \opn{Isom}_T(X_T,Y_T),
  \end{align*}
  for every $S$-scheme $T$, are represented by schemes locally of
  finite type over $S$.
\end{prop}
\begin{proof}
  Since the statement is local on $S$ by \cite[Theorem
  8.9]{GortzWed:AlgGeo}, we can assume that $S$ is Noetherian, and the
  proposition follows from \cite[Section 4.c,
  p.221-19f.]{FGA221}
\end{proof}

\begin{cor}\label{cor:isom_repr}
  Let $S$ be locally Noetherian, and let $\func {f_1}{X_1}Y$ and
  $\func{f_2}{X_2}{Y}$ be morphisms of projective $S$-schemes. Suppose
  that $X_1$ and $X_2$ are flat over $S$. Let \[\func{\calI:= \calI
    som_S((X_1,f_1),(X_2,f_2))}
  {({\mathbf{Sch}}_S)^\circ}{\mathbf{Sets}}\] be the functor defined
  by $$\calI(T)=\{\alpha\in \calI som_S(X_1,X_2)(T)\mid
  f_{1,T}=f_{2,T}\circ \alpha\}$$ for every $S$-scheme $T$. Then
  $\calI$ is represented by a scheme locally of finite type over $S$.
\end{cor}
\begin{proof}
  Let $\func\Delta{\calH om_S(X_1,Y)}{\calH om_S(X_1,Y)\times\calH
    om_S(X_1,Y)}$ be the diagonal map. Consider further the
  map \[\func \tau{\calI som_S(X_1,X_2)}{\calH om_S(X_1,Y)\times \calH
    om_S(X_1,Y)}\] given by $\tau(T)(\alpha)=(f_{1,T},f_{2,T}\circ
  \alpha)$. Then the diagram
  $$\xymatrix{\calI \ar[r] \ar[d] & 
    \calI som_S(X_1,X_2) \ar[d]^\tau \\
    \calH om_S(X_1,Y) \ar[r]^-\Delta & \calH om_S(X_1,Y)\times \calH
    om_S(X_1,Y)}$$ is Cartesian. Since the functors $\calI
  som_S(X_1,X_2)$ and $\calH om_S(X_1,Y)$ are representable by
  Proposition~\ref{prop:mor_isom_repres}, it follows that the fiber
  product $\calI$ is representable. Note moreover that $\Delta$ and
  hence $\calI \to \calI som_S(X_1,X_2)$ is locally of finite type.
  As $\calI som_S(X_1,X_2)$ is locally of finite type over $S$, it
  follows now that $\calI$ is locally of finite type over $S$.
\end{proof}

\begin{lm}\label{lm:isom_repr_mono}
  Let $S$ be a locally Noetherian scheme, and let $(C_1,i_1)$ and
  $(C_2,i_2)$ be elements of $\CM(S)$. Let $\calI=\calI
  som_S((C_1,i_1),(C_2,i_2))$ be the functor defined in
  Corollary~\ref{cor:isom_repr}. Then $\calI$ is represented by a
  scheme $Y$ such that the structure morphism $Y\ra S$ is a
  monomorphism locally of finite type.
\end{lm}
\begin{proof}
  Let $Y$ be the $S$-scheme that represents $\calI$ by
  Corollary~\ref{cor:isom_repr}, and let $\func h Y S$ be the
  structure morphism. Let $\func{g_1,g_2}T Y$ be two morphisms from a
  scheme $T$ such that $h\circ g_1=h\circ g_2$.  Taking this
  composition as structure map over $S$, we can consider $g_1$ and
  $g_2$ as elements of $\Hom_S(T,Y)=\calI(T)$. So $g_1$ and $g_2$
  correspond to two isomorphisms
  $\iso{\alpha_1,\alpha_2}{C_{1,T}}{C_{2,T}}$ over $\bbP^n_T$. Since
  every set $\calI(T)$ contains at most one element by
  Corollary~\ref{cor:auto}, it follows that $\alpha_1=\alpha_2$.  But
  this means $g_1=g_2$, and we get that $h$ is a monomorphism.
\end{proof}

\begin{prop}\label{prop:representable}
  The diagonal $\CM\ra \CM\times \CM$ is representable and locally of
  finite type.
\end{prop}
\begin{proof}
  Let $S$ be a locally Noetherian scheme, and \mbox{$S\to \CM\times
    \CM$} a map corresponding to elements $(C_1,i_1),(C_2,i_2)\in
  \CM(S)$. We have to show that the fiber product
  $\CM\times_{\CM\times \CM} S$ is representable and that the second
  projection \mbox{$\CM\times_{\CM\times \CM} S\to S$} is locally of
  finite type.
  
  For every morphism $\func g T S$, the composition $T\stackrel g \to
  S \to \CM\times \CM$ is represented by the elements
  $(C_{1,T},i_{1,T})$ and $(C_{2,T},i_{2,T})$ of $\CM(T)$.  It follows
  that
  $$(\CM\times_{\CM\times \CM}S)(T)=\left\{\func g T S\left\vert
 \begin{array}{ll} (C_{1,T},i_{1,T})=(C_{2,T},i_{2,T}) \\\text{ in
   }\CM(T)   \end{array}
\right.\right\}.$$ By
  definition of $\CM$, the two pairs $(C_{1,T},i_{1,T})$ and
  $(C_{2,T},i_{2,T})$ are equal if there exists an isomorphism
  $\iso\alpha{C_{1,T}}{C_{2,T}}$ over $T$ such that
  \mbox{$i_{1,T}=i_{2,T}\circ \alpha$}. In other words, we have that
  $$(\CM\times_{\CM\times \CM}S)(T)=\{\func g T S \mid \calI
  som_S((C_1,i_1),(C_2,i_2))(T)\neq \emptyset\}.$$ Note that the last
  expression makes sense since we can consider $T$ as an $S$-scheme
  via $g$. By Lemma~\ref{lm:isom_repr_mono}, the
  functor $\calI som_S((C_1,i_1),(C_2,i_2))$ is represented by a
  scheme $Y$, and the structure morphism $Y\to S$ is a
  monomorphism. Note that a morphism $\func g T S$ factors through $Y$
  if and only if $\Hom_S(T,Y)\neq \emptyset$. As $Y\to S$ is a
  monomorphism, such a factorization is unique and
  hence $$(\CM\times_{\CM\times \CM} S)(T)=\Hom(T,Y).$$ This shows that
  the fiber product $\CM\times_{\CM\times \CM}S$ is represented by the
  scheme $Y$ that is locally of finite type over $S$.
\end{proof}

Representability of the diagonal is equivalent to any map from a
scheme being representable.
\begin{prop}
  Let $\func F{{\mathbf{Sch}}^{\circ}}{\mathbf{Sets}}$ be a
  functor. Then the following properties are equivalent.
  \begin{enumerate}
  \item The diagonal $F\ra F\times F$ is representable.
  \item For every scheme $U$ and every $\xi\in F(U)$ the map
    $\func{\xi}{ U}F$ is representable.
  \end{enumerate}
\end{prop}
\begin{proof}
  Suppose first that the diagonal $F\ra F\times F$ is representable,
  and let $\func \xi {U} F$ and $\func \eta {V} F$ be two maps from
  schemes $U$ and $V$. We have to show that the fiber product
  $U\times_F V$ is representable.  Consider the map
  $\func{\xi\times\eta }{U\times V}{F\times F}$.  Then the fiber
  product $(U\times V)\times_{F\times F}F$ is representable by
  assumption. As $(U \times V) \times_{F\times F}F\cong U\times_F
  V$, we see that assertion (i) implies assertion~(ii).

  Suppose conversely that assertion (ii) holds. Let
  $\func{(\xi,\xi')}{ V}{F\times F}$ be a map from a scheme $V$ to the
  product $F\times F$. By assumption, the fiber product
  $V\times_F V$ with respect to the maps $\xi$ and $\xi'$
  is represented by a scheme $W$. Moreover, one checks that the
  diagram $$\xymatrix{W\times_{V\times V}V\ar[r] \ar[d] & V \ar[d]\\
    F\ar[r] & F\times F}$$ is Cartesian. This shows that the diagonal
  is representable, and hence assertion~(i) holds.
\end{proof}

\begin{cor}\label{cor:loc_fin_type}
  Let $\func \xi X \CM$ be a map, where $X$ is a scheme locally of finite
  type over $\spec(\bbZ)$. Then $\xi$ is representable and locally of
  finite type.
\end{cor}
\begin{proof}
  For any map $\func \eta Y \CM$ the diagram
  $$\xymatrix{X\times_{\CM} Y\ar[r] \ar[d] & X\times Y
    \ar[d]^{\xi\times\eta} \\ \CM\ar[r]^-\Delta & \CM\times \CM}$$ is
  Cartesian. Since the diagonal $\Delta$ is representable and locally
  of finite type, we see that $X\times_{\CM}Y\to X\times Y$ is locally
  of finite type. Now the statement follows as the projection $X\times
  Y\to Y$ is locally of finite type by assumption on $X$.
\end{proof}

\section{Embedding in projective space}

The goal of this section is to show that there exists a positive
integer $N$ such that for every base scheme $S$ and any $(C,i)\in
\CM(S)$ there exists an open cover $\{S_a\}$ of $S$ such that
$C_{S_a}$ can be embedded into the projective space $\bbP^N_{S_a}$.

First we consider the case of curves over fields. We show in
Proposition~\ref{prop:field_closed_imm} that there exist integers $m$
and $N:=N(m)$ such that for every field $k$ and every $(C,i)\in
\CM(\spec(k))$ the invertible sheaf $\calL:=i^*\calO_{\bbP^n_k}(m)$ is
very ample and every choice of basis of $H^0(C,\calL)$ gives rise to a
closed immersion $\injfunc j C{\bbP^N_k}$.

The main tool here is {\em Castelnuovo--Mumford regularity} that also
plays an important roll in the construction of the Hilbert scheme. In
Subsection~\ref{sec:m-reg}, we give an introduction to this concept.

The case of a general base scheme $S$ is then treated in
Subsection~\ref{sec:emb_S}.

\subsection{Some facts about Castelnuovo--Mumford
  regularity} \label{sec:m-reg}
In the following, we sometimes suppress the projective space
$\bbP^n_k$ in the notation of the cohomology groups $H^i(\bbP^n_k,
\calF)$ for a coherent sheaf $\calF$ in order to increase
readability. Moreover, we write $h^i(\calF):=\dim_k H^i(\calF)$ for
its dimension.
\begin{defi}
  Let $\bbP^n_k$ be the projective $n$-space over a field $k$, and let
  $\calF$ be a coherent sheaf on $\bbP^n_k$. Let $m\in \bbZ$ be an
  integer. If $$H^i(\bbP^n_k,\calF(m-i))=0$$ for every $i>0$, then
  $\calF$ is $m$-{\em regular}.
\end{defi}

\begin{ex}\label{ex:m-reg_0}
  Suppose that $\calF$ is a coherent sheaf on $\bbP^n_k$ with
  zero-dimen\-sional support. Then $\calF$ is $m$-regular for every
  $m\in \bbZ$ since the cohomology vanishes in degrees that exceed the
  dimension of the support.
\end{ex}   
\begin{ex}\label{ex:m-reg_bbP}
  For every integer $d$, the invertible sheaf $\calO(d):=
  \calO_{\bbP^n_k}(d)$ on $\bbP^n_k$ is $(-d)$-regular.  Indeed, we have that
  $H^i(\calO(d+m-i))=0$ unless $i=0,n$ and
  $H^n(\calO(d+m-n))\cong H^0(\calO(-d-m-1))^\vee$ by
  \cite[Theorem III.5.1]{Hartshorne}.
  Clearly the last term vanishes for $m\geq -d$.
\end{ex}

\begin{prop}[{\cite[Proposition 1.3]{Kleiman:m-reg}}]\label{prop:m-reg}
  Let $\calF$ be a $m$-regular coherent sheaf on $\bbP^n_k$. Then for
  every $m'\geq m$ we have the following.
  \begin{enumerate}
  \item\label{prop:m-reg1} $\calF$ is $m'$-regular.
  \item\label{prop:m-reg2} $\calF(m')$ is generated by its global
    sections.
  \end{enumerate}
\end{prop}
\begin{cor}\label{cor:m-reg}
  Let $\calF$ be $m$-regular. Then $H^i(\bbP^n_k,\calF(m))=0$ for
  $i>0$, and hence $h^0(\calF(m))=p_\calF(m)$ for the
  Hilbert polynomial $p_\calF(t)$ of $\calF$.
\end{cor}
\begin{proof}
  As $H^i(\bbP^n_k,\calF(m))=H^i(\bbP^n_k,\calF((m+i)-i))$, the
  statement is a direct consequence of assertion (i) in the
  proposition.
\end{proof}

\begin{lm}\label{lm:m-reg_ex_seq}
  Let $0\ra \calF\ra \calG\ra \calH\ra 0$ be an exact sequence of
  coherent sheaves on $\bbP^n_k$. Then we have the following.
  \begin{enumerate}
  \item\label{lm:m-reg_ex_seq1} If $\calF$ and $\calH$ are
    $m$-regular, then $\calG$ is $m$-regular.
  \item\label{lm:m-reg_ex_seq2} Suppose that $\calG$ is $m$-regular
    and that $H^i(\bbP^n_k, \calF(m+1-i))=0$ for all $i>1$ (in
    particular, this is satisfied if $\calF$ is $(m+1)$-regular). Then
    $\calH$ is $m$-regular.
  \end{enumerate}
\end{lm}
\begin{proof}
  The statements follow directly from the induced long exact sequence
 \begin{eqnarray}\label{eqn:m-reg_ex_seq}
   \cdots\ H^i(\calF(d))\longrightarrow
   H^i(\calG(d)) \longrightarrow H^i(\calH(d))\longrightarrow
   H^{i+1}(\calF(d))\ \cdots
 \end{eqnarray}
in cohomology for $d=m-i$.  

For assertion~\ref{lm:m-reg_ex_seq1} suppose that $\calF$ and
$\calH$ are $m$-regular. Then we have that
$H^i(\calF(m-i))=H^i(\calH(m-i))=0$ for all $i> 0$. From
sequence~(\ref{eqn:m-reg_ex_seq}) we get that $H^i(\calG(m-i))=0$ for
all $i> 0$, and $\calG$ is $m$-regular.

Moreover, if $H^i(\calG(m-i))=H^{i+1}(\calF(m+1-(i+1)))=0$ for $i>0$,
then $H^i(\calH(m-i))=0$, and hence assertion \ref{lm:m-reg_ex_seq2}.
\end{proof}

\begin{thm}[{\cite[Theorem in Lecture 14]{Mumford:Lectures}}]
  \label{thm:ideals_m-reg} For every numerical polynomial $p(t)\in
  \bbQ[t]$ there exists an integer $m$ such that all coherent sheaves
  of ideals on $\bbP^n_k$ having Hilbert polynomial $p(t)$ are
  $m$-regular.
\end{thm}

\begin{prop}\label{prop:kernel_m-reg}
  Let $m'\in\bbN$, and let $p(t)\in \bbQ[t]$ be
  a numerical polynomial. Then there exists an integer $m=m(m',p)\geq
  m'$, such that for every short exact sequence
  \begin{equation} \xymatrix{0\ar[r] & \calF \ar[r] & \calG \ar[r] &
      \calH \ar[r] & 0}\label{eqn:ex_seq0}
  \end{equation}
  of coherent sheaves on $\bbP_k^n$ with $\calG$ being $m'$-regular
  and $\calH$ having Hilbert polynomial $p(t)$, the sheaf $\calF$ is
  $m$-regular.
\end{prop}
\begin{proof}
  Since $m$-regularity does not depend on the field $k$, we can
  without loss of generality assume that $k$ is infinite.

  We prove the proposition by induction on $n$. For $n=0$, there is
  nothing to show, and we can directly assume that $n>0$. Since $k$ is
  infinite, there exists a hyperplane $j\mathpunct: H \subseteq
  \bbP^n_k$ such that in the diagram
  \begin{equation}\label{eqn:KFG}
    \begin{array}{rcl}
      \xymatrix{ & 0\ar[d] &0\ar[d] &0\ar[d] & \\ 0\ar[r]
        & \calF(-1) \ar[r]\ar[d] & \calF \ar[r] \ar[d] & j_*(\calF\vert_H)
        \ar[r]\ar[d] &0\\0\ar[r] & \calG(-1)\ar[d] \ar[r] & \calG
        \ar[r]\ar[d] & j_*(\calG\vert_H) \ar[r]\ar[d] &0\\0\ar[r] &
        \calH(-1) \ar[d]\ar[r] & \calH \ar[r] \ar[d] & j_*(\calH\vert_H)
        \ar[r]\ar[d] &0 \\ & 0 & 0 & 0 &}
    \end{array}
  \end{equation}
  all sequences are exact. Indeed, take $H$ so that is does not
  contain any of the associated points of $\calF$, $\calG$ and
  $\calH$. Then all horizontal and the first two vertical sequence are
  exact. Exactness of the third vertical sequence follows by diagram
  chase. Moreover, we get an induced short exact
  sequence $$\xymatrix{0\ar[r] & \calF\vert_H\ar[r] &
    \calG\vert_H\ar[r] & \calH\vert_H\ar[r] & 0}$$ on
  $H\cong\bbP^{n-1}_k$. Note that also $\calG\vert_H$ is $m'$-regular
  by Lemma~\ref{lm:m-reg_ex_seq}\ref{lm:m-reg_ex_seq2}. Furthermore,
  the Hilbert polynomial $q(t)=p(t)-p(t-1)$ of $\calH\vert_H$ depends
  only on $p(t)$. By induction, there exists $m_1=m_1(m',q)\in\bbN$
  such that the restriction $\calF\vert_H$ is $m_1$-regular. This
  implies by \cite[Proposition 1.4]{Kleiman:m-reg} that $\calF$ is
  $(m_1+h^1(\calF(m_1-1)))$-regular, and it remains to show that there
  exists an independent bound for $h^1(\calF(m_1-1))$.

  Note that by Proposition~\ref{prop:m-reg}\ref{prop:m-reg1}, we can
  assume that $m_1\geq m'$. In particular, the sheaf $\calG$ is
  $m_1$-regular, and hence
  $H^1(\calG(m_1-1))=0$. Sequence~(\ref{eqn:ex_seq0}) implies that the
  map $H^0(\calH(m_1-1))\to H^1(\calF(m_1-1))$ is surjective, and hence
  $h^1(\calF(m_1-1))\leq h^0(\calH(m_1-1))$. Finally, we claim that
  $\calH$ is $m_1$-regular. This concludes the proof since then
  $h^0(\calH(m_1-1))=p(m_1-1)$, and we can set $m:=m_1+p(m_1-1)$.

  To prove the claim, we first observe that, by
  Lemma~\ref{lm:m-reg_ex_seq}\ref{lm:m-reg_ex_seq2}, it suffices to
  show that $H^i(\calF(m_1+1-i))=0$ for all $i>1$. To simplify
  notation, we set $\calF':=j_*(\calF\vert_H)$. After twisting the
  short exact sequence $$\xymatrix{0\ar[r] &\calF(-1)\ar[r]
    &\calF\ar[r] &\calF'\ar[r] &0}$$ from diagram~\eqref{eqn:KFG} by
  $\calO(d)$, we consider the induced long exact sequence
\begin{align*}
  \cdots\ H^{i-1}(\calF'(d))\longrightarrow
  H^i(\calF(d-1))\longrightarrow H^i(\calF(d))\longrightarrow
  H^i(\calF'(d))\ \cdots
  \end{align*} 
  in cohomology.  Let $i>1$. Then we have
  \mbox{$H^{i-1}(\calF'(d))=H^i(\calF'(d))=0$} for $d\geq m_1+1-i$
  since the sheaf $\calF'$ is $m_1$-regular. Therefore it follows that
  \mbox{$H^i(\calF(d-1))\cong H^i(\calF(d))$} for all $d\geq
  m_1+1-i$. In particular, we see that $H^i(\calF(m_1-i))\cong
  H^i(\calF(m_1-i+1))\cong H^i(\calF(m_1-i+2))\cong \ldots$. But
  $H^i(\calF(d))=0$ for $d\gg 0$, and we obtain
  $H^i(\calF(m_1+1-i))= 0$.
\end{proof}

\subsection{Embedding over a field}

The results on Castelnuovo--Mumford regularity can then be used to
show that there exist integers $m$ and $N$ such that for every field
$k$ and every $(C,i)\in \CM(\spec(k))$, the invertible sheaf
$i^*\calO_{\bbP^n_k}(m)$ is generated by $N+1$ global sections and the
corresponding map $\func j C \bbP^n_k$ is a closed immersion.

\begin{prop}\label{prop:univ_m-reg}
  There exists $m'\in \bbN$ such that for every field $k$ and every
  pair $(C,i)\in \CM(\spec(k))$ the coherent sheaf $i_*\calO_C$ on
  $\bbP^n_k$ is $m'$-regular.
\end{prop}
\begin{proof}
  Consider the factorization $$\xymatrix{C\ar[rr]^i\ar[dr]^-{h} & &
    \bbP^n_k, \\ & i(C)\ar@{^(->}[ur]^-j & }$$ through the
  scheme-theoretic image $i(C)$, where $h$ is finite and an
  isomorphism away from finitely many closed points, and $j$ is a
  closed immersion.  Let $\calJ$ and $\calK$ be the kernel and the
  cokernel of the natural map $\calO_{\bbP^n_k}\ra i_*\calO_C$.  Then
  we have short sequences
  \begin{equation}\label{eqn:m-reg1}
    \xymatrix{0\ar[r] & \calJ\ar[r] & \calO_{\bbP^n_k}\ar[r] &
      j_*\calO_{i(C)}\cong \calO_{\bbP^n_k}/\calJ\ar[r] & 0} \end{equation}
  and
  \begin{equation}\label{eqn:m-reg2}
    \xymatrix{0\ar[r] & j_*\calO_{i(C)}\ar[r] & i_*\calO_C\ar[r] &
      \calK\ar[r] & 0}
  \end{equation}
  of $\calO_{\bbP^n_k}$-modules. We have seen in
  Proposition~\ref{prop:image_Hilb_poly} that there are only finitely
  many possibilities for the Hilbert polynomial of
  $i(C)$. Consequently, by the exact sequence
  \eqref{eqn:m-reg1}, there are only finitely many possibilities for
  the Hilbert polynomial of $\calJ$. By
  Theorem~\ref{thm:ideals_m-reg}, there exists an integer $m'$
  independent of $k$, $C$ and $i$ such that $\calJ$ is
  $m'$-regular. Consider sequences~\eqref{eqn:m-reg1} and
  \eqref{eqn:m-reg2}, and recall that $\calO_{\bbP^n_k}$ and $\calK$
  are $0$-regular by Example~\ref{ex:m-reg_bbP} and \ref{ex:m-reg_0}
  respectively. We see with Lemma~\ref{lm:m-reg_ex_seq} that also
  $i_*\calO_C$ is $m'$-regular.
\end{proof} 
\begin{prop}\label{prop:point_kernel_m-reg}
  Let $m'$ be as in Proposition~\ref{prop:univ_m-reg}. There exists an
  integer $m\geq m'$ such that for every field $k$, every $(C,i)\in
  \CM(\spec(k))$ and every closed $k$-point $x\in C$, the
  sheaf $i_*\calJ$ is $m$-regular, where $\calJ$ is the sheaf of
  ideals on $C$ corresponding to the closed immersion
  \mbox{$\injfunc{h}{\spec(\kappa(x))}{C}$}.
\end{prop}
\begin{proof}
  Let $k$ be a field, $(C,i)\in \CM(\spec(k))$ and $x\in C$ be a
  $k$-rational point. The closed immersion $h$ gives rise to a short
  exact sequence
  $$\xymatrix{0\ar[r] & \calJ\ar[r] & \calO_C\ar[r] &
    h_*\calO_{\spec(\kappa(x))}\ar[r] &0}$$ of
  $\calO_C$-modules. Since $i$ is finite and hence in particular
  affine, also the induced sequence
  $$\xymatrix{0\ar[r] & i_*\calJ\ar[r] & i_*\calO_C\ar[r] &
    i_*h_*\calO_{\spec(\kappa(x))}\ar[r] &0}$$ on $\bbP^n_k$ is
  exact. Let $m'\in \bbN$ be as in
  Proposition~\ref{prop:univ_m-reg}. Then the coherent sheaf
  $i_*\calO_C$ is $m'$-regular. Moreover, the coherent sheaf
  $i_*h_*\calO_{\spec(\kappa(x))}$ has Hilbert polynomial constant
  equal to $1$ since $x$ is a $k$-rational point.  Then the statement is
  a direct consequence of Proposition~\ref{prop:kernel_m-reg}.
\end{proof}

\begin{lm}\label{lm:dir_im_gen}
  Let $\func f X Y$ be an affine morphism of schemes, and let $\calF$
  be a quasi-coherent sheaf on $X$. Suppose that the direct image
  $f_*\calF$ is generated by its global sections. Then $\calF$ is
  generated by its global sections.
\end{lm}
\begin{proof}
  Since $f_*\calF$ is generated by its global sections, there exists a
  surjective map $\bigoplus_{i\in I}\calO_Y\surj f_*\calF$. Applying
  the right exact inverse image functor $f^*$, we get a surjective map
  $\bigoplus_{i\in I}\calO_X\surj f^*f_*\calF$. Since $f$ is affine,
  the natural homomorphism $f^*f_*\calF\ra \calF$ is surjective, see
  \cite[Remarques (3.4.7)]{EGAII}. Composition of these two maps
  yields the required surjection.
\end{proof}

\begin{prop}\label{prop:point_glob_sec}
  Let $m$ be as in Proposition~\ref{prop:point_kernel_m-reg}. Let $k$
  be a field, and let $(C,i)\in \CM(\spec(k))$.
  \begin{enumerate}
  \item The coherent sheaf $i^*\calO_{\bbP^n_k}(m)$ is generated by
    its global sections, its higher cohomology vanishes and
    $h^0(i^*\calO_{\bbP^n_k}(m))=p(m)$.
  \item\label{prop:point_glob_sec2} Let $x\in C$ be a $k$-rational
    point, and let $\calJ$ be the sheaf of ideals of the closed
    immersion $\injfunc{h}{\spec(\kappa(x))}{C}$. Then the coherent
    sheaf $\calJ\otimes_{\calO_C}i^*\calO_{\bbP^n_k}(m)$ is generated
    by its global sections and its higher cohomology vanishes.
  \end{enumerate}
\end{prop}
\begin{proof}
  Let $\calF$ denote the structure sheaf $\calO_C$ or the sheaf of ideals
  $\calJ$. Note that in both cases the direct image $i_*\calF$ is
  $m$-regular by the choice of $m$. Hence, by
  Proposition~\ref{prop:m-reg} and Corollary~\ref{cor:m-reg}, the
  twist $(i_*\calF)(m)$ is generated by its global sections and its
  higher cohomology vanishes. Note further that $(i_*\calF)(m)\cong
  i_*(\calF\otimes_{\calO_C} i^*\calO_{\bbP^n_k}(m))$ by the
  projection formula. Thus the proposition follows from
  Lemma~\ref{lm:dir_im_gen} and the fact that we have
  $H^r(C,\calG)=H^r(\bbP^n_k,i_*\calG)$ for every coherent sheaf
  $\calG$ on $C$ and $r\geq 0$ since the morphism $i$ is affine.
\end{proof}

\begin{lm}\label{lm:sep_points_tangents}
  Let $k$ be an algebraically closed field, and let $X$ be a
  scheme over $k$. Let $\calL$ be an invertible sheaf on $X$ that is
  generated by its global sections. Suppose further that
  $\calJ\otimes_{\calO_X}\calL$ is generated by its global sections
  for every sheaf of ideals $\calJ$ corresponding to a closed point of
  $X$. Then the global sections of $\calL$ separate points and tangent
  vectors.
\end{lm}
\begin{proof}
  Let $x$ be a closed point of $X$, and let $\calJ$ be the sheaf of
  ideals of the closed immersion $\spec(\kappa(x))\hookrightarrow
  X$. 

  Let first $y\neq x$ be another closed point of $X$. Since the sheaf
  $\calJ\otimes_{\calO_X}\calL$ is globally generated, there exists a 
  section $s\in H^0(X,\calJ\otimes_{\calO_X}\calL)$ such that
  $s_y\notin \m_y(\calJ\otimes_{\calO_X}\calL)_y=\m_y\calL_y$. Note
  that $s_x\in (\calJ\otimes_{\calO_X}\calL)_x =
  \m_x\calL_x$. From the inclusion $H^0(X,\calJ\otimes_{\calO_X}\calL)
  \subseteq H^0(X,\calL)$, we get that the global sections of $\calL$
  separate points.

  It remains to show that the vector space
  $\m_x\calL_x/\m_x^2\calL_x$ is spanned by the set $\{ s \in
  H^0(C,\calL)\mid s_x\in \m_x\calL_x\}$. Since
  $\m_x\calL_x=(\calJ\otimes_{\calO_X} \calL)_x$, this is a direct
  consequence of the sheaf $\calJ\otimes _{\calO_X}\calL$ being
  generated by its global sections.
\end{proof}
\begin{prop}\label{prop:field_closed_imm}
  Let $k$ be a field, and let $(C,i)\in \CM(\spec(k))$. Let $m$ be as
  in Proposition~\ref{prop:point_kernel_m-reg}, and set
  $N:=p(m)-1$. Then the sheaf $i^*\calO_{\bbP^n_k}(m)$ is very
  ample. In particular, there exists a closed immersion $\injfunc j C
  \bbP^N_k$ such that
  \mbox{$j^*\calO_{\bbP^N_k}(1)=i^*\calO_{\bbP^n_k}(m)$}.
\end{prop}
\begin{proof}
  By Proposition~\ref{prop:point_glob_sec}, the invertible sheaf
  $\calL:=i^*\calO_{\bbP^n_k}(m)$ is generated by its global sections
  and $h^0(C,\calL)=p(m)$.  Let $s_0,\dotsc,s_N$ be a basis of
  $H^0(C,\calL)$. We claim that the map $\func j C\bbP^N_k$
  corresponding to $\calL$ and the sections $s_0,\dotsc,s_N$ is a
  closed immersion.  Note that $j$ is a closed immersion if and only
  if the induced map $\func{\bar j}{C_{\bar k}}{\bbP^N_{\bar k}}$,
  obtained by base change to the algebraic closure $\bar k$ of $k$, is
  a closed immersion. As we moreover have that $H^0(C,\calL)\otimes_k
  \bar k=H^0(C_{\bar k},\calL_{\bar k})$, we can without loss of
  generality assume that $k$ is algebraically closed.

  The invertible sheaf $\calL$ satisfies the conditions of
  Lemma~\ref{lm:sep_points_tangents} by choice of $m$, and hence its
  global sections separate points and tangent vectors. Now it follows
  from \cite[Proposition II.7.3]{Hartshorne} that $\bar j$ is a closed
  immersion.
\end{proof}

\subsection{Embedding over a general base scheme}\label{sec:emb_S}
Now let $S$ be a locally Noetherian scheme, and let $N\in \bbN$ be as
in the previous subsection. We show that for every pair $(C,i)\in
\CM(S)$, the scheme $S$ has an open cover $\{S_a\}$ such that for
every $a$ there exists a closed immersion $C_{S_a}\hookrightarrow
\bbP^{N}_{S_a}$.

\begin{thm}[Cohomology and base change]\label{thm:cohom_basechange}
  Let $S$ be a locally Noetherian scheme, and let $\func f X S$ be a
  proper morphism. Let $\calF$ be a coherent $\calO_X$-module that is
  flat over $S$. Suppose that $H^r(X_s,\calF_s)=0$ for all $s\in S$
  and $r>0$. Then we have the following.
  \begin{enumerate}
  \item The $\calO_S$-module $f_*\calF$ is locally free.
  \item For every morphism $\func g T S$ the natural base change map
    $g^*f_*\calF\ra (f_T)_*\calF_T$ is an isomorphism.
  \end{enumerate}
\end{thm}
\begin{proof}
  The first statement (i) is \cite[Corollaire (7.9.9)]{EGAIII2}. For
  statement (ii) consider \cite[Corollaire (6.9.9)]{EGAIII2} in the
  special case that $\calP_0=\calF$ and $\calP_i=0$ for $i\neq 0$.
\end{proof}
\begin{prop}\label{prop:LocFree_BaseChange}
  Let $S$ be locally Noetherian. Let $(C,i)\in \CM(S)$,
  and denote by $\func f C S$ the structure map. Let further
  $\calL:=i^*\calO_{\bbP^n_S}(m)$ with $m$ as in
  Proposition~\ref{prop:point_kernel_m-reg}. Then the following holds.
  \begin{enumerate}
  \item The sheaf $f_*\calL$ is locally free of rank $p(m)$.
  \item For every morphism $\func g T S$ the natural base change map
    $g^*f_*\calL\ra (f_T)_*\calL_T$ is an iso\-morphism.
  \end{enumerate}
\end{prop}
\begin{proof}
  For every $s\in S$, we have $\calL_s=(i_s)^*\calO_{\bbP^n_{\kappa(s)}}(m)$,
  and hence \mbox{$H^r(C_s,\calL_s)=0$} for all $r> 0$ by
  Proposition~\ref{prop:point_glob_sec}. Then the statement follows
  directly from Theorem~\ref{thm:cohom_basechange}.
\end{proof}

\begin{prop}\label{prop:star_star_surj}
  Let $S$ be a locally Noetherian scheme, and let $(C,i)\in
  \CM(S)$. With the notation as in
  Proposition~\ref{prop:LocFree_BaseChange}, the natural map $\func
  \pi {f^*f_*\calL} \calL$ is surjective.
\end{prop}
\begin{proof}
  The cokernel $\calN$ of $\pi$ is coherent, and it suffices to show
  that $\calN_s=0$ for every $s\in S$.

  So let $s\in S$, and consider the fiber diagram
  $$\xymatrix{C_s\ar[d]_{f_s}\ar[r]^h & C \ar[d]^f \\
    \spec(\kappa(s))\ar[r]^-g & S.}$$ Since $\calL_s=h^*\calL$ is
  generated by its global sections by assumption on $m$, see
  Proposition~\ref{prop:point_glob_sec}, the natural map $\func
  {\pi_s}{f_s^*(f_s)_*\calL_s}\calL_s$ is surjective. Note that the
  map $\func {h^*\pi}{h^*f^*f_*\calL}{h^*\calL}$ factors as
  $$h^*f^*f_*\calL\cong f_s^*g^*f_*\calL
  \stackrel{\alpha}\longrightarrow f_s^*(f_s)_*h^*\calL
  \stackrel{\pi_s}\longrightarrow h^*\calL,$$ where $\alpha$ is the
  pullback of the natural map $g^*f_*\calL\ra(f_s)_*h^*\calL$ which is
  an isomorphism by Proposition~\ref{prop:LocFree_BaseChange}. It
  follows that also $h^*\pi$ is surjective, that is, $\calN_s=0$.
\end{proof}

\begin{prop}[{\cite[Proposition {(4.6.7)}]{EGAIII1}}]
  \label{prop:EGA_finite_closed_fiber} 
  Let $S$ be a locally Noetherian scheme, and let $\func f X Y$ be a
  morphism of proper $S$-schemes. Suppose that $\func{f_s}{X_s}{Y_s}$
  is finite (resp.\ a closed immersion) for a point $s\in
  S$. Then there exists an open neighborhood $U\subseteq S$ of $s$
  such that the restriction $\func{f_U}{X_U}{Y_U}$ is finite
  (resp.\ a closed immersion).
\end{prop}

\begin{prop}\label{prop:isom-embedding}
  Let $(C,i) \in \CM(S)$, where $S$ is a locally Noetherian scheme.
  Let $m$ be as in Proposition~\ref{prop:point_kernel_m-reg}, and set
  $\calL:=i^*\calO_{\bbP^n_S}(m)$. Suppose that $f_*\calL$, where
  $\func f C S$ denotes the structure map, is a free
  $\calO_S$-module of rank $N+1$. Then every isomorphism $\iso\sigma
  {\calO_{S}^{N+1}}{f_*\calL}$ gives rise to a closed immersion
  $\injfunc{j_\sigma}{C}{\bbP^N_S}$ with
  $j_\sigma^*\calO_{\bbP^N_S}(1)=\calL$.
\end{prop}
\begin{proof}
  Composing the pull-back $f^*\sigma$ with the map $f^*f_*\calL \surj
  \calL$, that is surjective by Proposition~\ref{prop:star_star_surj},
  gives rise to a $S$-morphism $\func {j_\sigma} {C}{\bbP^N_S}$ such
  that $j_\sigma^*\calO_{\bbP^N_S}(1)=\calL$. We observe that
  $j_\sigma$ is a closed immersion over every fiber by
  Proposition~\ref{prop:field_closed_imm}.  Then it follows from
  Proposition~\ref{prop:EGA_finite_closed_fiber} that $j_\sigma$
  itself is a closed immersion.
\end{proof}
\begin{thm}\label{thm:embedding_general_base}
  Let $m$ be as in Proposition~\ref{prop:point_kernel_m-reg} and set
  $N:=p(m)-1$. For any locally Noetherian scheme $S$, and any
  $(C,i)\in \CM(S)$, the invertible sheaf
  $\calL:=i^*\calO_{\bbP^n_S}(m)$ is very ample for the structure
  morphism $\func f C S$.

  Moreover, the scheme $S$ has an open cover $\{S_a\}_{a\in I}$ such
  that for every $a\in I$ there exists a closed immersion $j_a\!
  :C_a:=C\times_S S_a\hookrightarrow \bbP^N_{S_a}$ with
  $j_a^*\calO_{\bbP^N_{S_a}}(1)=i_a^*\calO_{\bbP^n_{S_a}}(m)$.
\end{thm}
\begin{proof}
  By Proposition~\ref{prop:LocFree_BaseChange}, the coherent sheaf
  $f_*\calL$ is locally free of rank $N+1$. Let $\{S_a\}_{a\in I}$ be
  an open cover of $S$ such that every restriction
  $(f_*\calL)\vert_{S_a}$ is free. For every $a\in I$ consider the
  induced element $(C_a,i_a)$ in $\CM(S_a)$ with structure map
  $\func{f_a}{C_a}{S_a}$. According to
  Proposition~\ref{prop:isom-embedding}, every choice of basis for
  $(f_*\calL)\vert_{S_a}=(f_a)_*i_a^*\calO_{\bbP^n_{S_a}}(m)$ then gives
  a closed immersion $\injfunc {j_a} {C_a}{\bbP^N_{S_a}}$ with
  the required properties.
\end{proof}

\section{{\texorpdfstring{A covering of $\CM$}{A covering of
      CM}}}\label{chap:cover}

Let $m$ be as in Proposition~\ref{prop:point_kernel_m-reg}, and set
$N:=p(m)-1$. We show in Subsection~\ref{sec:sub_hilb} that there exists a
scheme $W_0$ parameterizing all closed subschemes $Z\subseteq
\bbP^n\times\bbP^N$ such that $(Z,\pr_1)$ is an element of $\CM$, and
the second projection $\func {\pr_2}Z\bbP^N$ is a closed
immersion. Here $\pr_1$ and $\pr_2$ denote the projections $Z\ra
\bbP^n$ and $Z\ra \bbP^N$ respectively. In
Subsection~\ref{sec:refinement}, we construct a refinement $W$ of $W_0$
corresponding to the closed subschemes $Z\subseteq\bbP^n\times\bbP^N$ as
above such that the second projection $\pr_2$ is given by the
invertible sheaf $\pr_1^*\calO_{\bbP^n}(m)$. Moreover, we get a
surjective map $\func \pi W \CM$ that maps $Z$ to $(Z,\pr_1)$.

\subsection{Subscheme of the Hilbert scheme}\label{sec:sub_hilb}

Let $P(t)=p((m+1)t)$, and let $H:=\Hilb_{\bbP^n\times \bbP^N}^{P(t)}$
be the Hilbert scheme parameterizing closed subschemes of
$\bbP^n\times\bbP^N$ having Hilbert polynomial $P(t)$ with respect to
the very ample sheaf $\calO(1,1)$. In the following, we show that the functor
$\func F{(\mathbf{Sch})^\circ}{\mathbf{Sets}}$ given
by $$F(S):=\left\{Z\subseteq \bbP^n_S\times_S\bbP^N_S\text{ in
  }H(S)\left\vert \begin{array}{l} (Z,\pr_1)\in \CM(S) \text{ and }\\
      \pr_2\text{ is a closed immersion}\end{array}\right.\right\},$$
where $\func {\pr_1}{Z}{\bbP^n_S}$ and $\func{\pr_2}{Z}{\bbP^N_S}$
denote the projections, is represented by an open subscheme $W_0$ of
$H$.

\subsubsection{Properties of the projections } 
\begin{prop}\label{prop:finite_closed_fiber}
  Let $S$ be locally Noetherian, and let $\func f X Y$ be a
  morphism of proper $S$-schemes. There exists an open subscheme
  $U$ of $S$ such that a morphism $\func g T S$ factors through $U$
  if and only if the morphism $\func{f_T}{X_T}{Y_T}$ obtained by base
  change is finite (resp. a closed immersion).
\end{prop}
\begin{proof}
  Let $\calP$ denote one of the properties ``is a closed immersion''
  and ``is finite''.  

  By Proposition~\ref{prop:EGA_finite_closed_fiber}, the set
  $U:=\{s\in S\mid \func{f_s}{X_s}{Y_s}\text{ has }\calP\}$
  is open, and the restriction $\func{f_U}{X_U}{Y_U}$ has property
  $\calP$. Since property $\calP$ is stable under base change, a
  morphism $T\to S$ that factors through $U$ has $\calP$.

  Now consider a morphism $\func g T S$ such that the induced map
  $f_T$ has property $\calP$. Then $\func{(f_T)_t}{(X_T)_t}{(Y_T)_t}$
  has $\calP$ for every $t\in T$. As the morphism
  \mbox{$\func{f_{g(t)}}{X_{g(t)}}{Y_{g(t)}}$} has $\calP$ if and only
  if the morphism $(f_T)_t$ obtained by the change of base field
  $\spec(\kappa(t))\to \spec(\kappa(g(t)))$ has $\calP$, it follows
  that $g(t)\in U$. Since $U$ is an open subscheme, we get that $g$
  factors through $U$.
\end{proof}

\subsubsection{Cohen--Macaulay and of pure dimension $1$}
  
\begin{prop}\label{prop:CM+equidim_fibers2}
  Let $\func f X S$ be a flat, proper morphism of finite
  presentation.  Then there exists an open subscheme $U$ of $S$ such
  that a morphism $\func g T S$ factors through $U$ if and only if
  all fibers of $X_T$ over $T$ are Cohen--Macaulay and of pure
  dimension $1$.
\end{prop}
\begin{proof}
  Note that a fiber $X_s$ for $s\in S$ is Cohen--Macaulay and of pure
  dimension $1$ if and only if the structure sheaf
  $\calO_{X_s}=(\calO_X)_s$ does not have any embedded point and all
  irreducible components of $\supp((\calO_X)_s)$ have dimension $1$.
  By \cite[Théorème (12.2.1)(iv)]{EGAIV3}, the set $U:=\{s\in S\mid
  X_s \text{ is Cohen--Macaulay and of pure dimension 1}\}$ is then
  open.

  Consider a morphism $\func g T S$, and let $t\in T$. Then $(X_T)_t$
  is obtained from $X_{g(t)}$ by base change $\spec(\kappa(t))\ra
  \spec(\kappa(g(t)))$. By Lemma~\ref{lm:func_field_change}, the fiber
  $(X_T)_t$ is Cohen--Macaulay and of pure dimension~$1$ if and only if
  $g(t)\in U$. The statement now follows as $U$ is open.
\end{proof}

\subsubsection{Isomorphism onto image away from finite set of closed
  points}
\begin{prop}\label{prop:isom_on_image_open}
  Let $\func h X Y$ be a finite morphism of schemes over a locally
  Noetherian scheme $S$. Suppose that $Y$ is proper over $S$. Then
  there exists an open subscheme $U$ of $S$ such that a morphism $T\to
  S$ factors through $U$ if and only if for every $t\in T$ the map
  $\func{h_t}{X_t}{Y_t}$ is an isomorphism onto its image away from
  finitely many closed points.
\end{prop}
\begin{proof}
  Consider the closed subset $Z=\supp(\coker(\calO_Y\to h_*\calO_X))$
  of $Y$ with the induced reduced scheme structure. Then {\em
    Chevalley's upper semicontinuity theorem} \cite[Théorème
  (13.1.5)]{EGAIV3} implies that the locus $U:=\{s\in S\mid \dim(Z_s)=0\}$ is
  open.

  By Lemma~\ref{lm:func_field_change}\ref{lm:func_field_change_isom},
  a morphism $\func g T S$ factors through the open subscheme $U$ if
  and only if $\dim((Z_{g(t)})_t)=0$ for every $t\in T$. As, moreover, 
  $\dim((Z_{g(t)})_t)=\dim(\supp(\coker(\calO_{Y_T}\to
  (h_T)_*\calO_{X_T}))_t)$, the statement follows with Lemma~\ref{lm:iso_alt}.
\end{proof}

\subsubsection{The Hilbert polynomial}
\begin{prop}\label{prop:Hilbert_poly_loc_const}
  Let $\func h X \bbP^n_S$ be a finite morphism of schemes over a
  locally Noetherian scheme $S$. Suppose that $X$ is flat over
  $S$. Let further \mbox{$p(x)\in \bbQ[x]$} be a numerical
  polynomial. Then there exists an open and closed subscheme $U$ of
  $S$ such that a morphism $\func g T S$ factors through $U$ if and
  only if the coherent sheaf $((h_T)_*\calO_{X_T})_t$ has Hilbert
  polynomial $p(x)$ for every $t\in T$.
\end{prop}
\begin{proof}
  Note that since $X$ is flat over $S$ and $h$ is affine, the direct
  image $h_*\calO_X$ is flat over $S$. In particular, its
  Hilbert polynomial is locally constant on $S$ by \cite[Proposition
  (7.9.11)]{EGAIII2}, and therefore the subset \mbox{$U:=\{s\in S\mid
    (h_*\calO_X)_s \text{ has Hilbert polynomial }p(x)\}$} is open and
  closed. Moreover, we have that $(h_T)_*\calO_{X_T}=(h_*\calO_X)_T$
  since $h$ is affine. It follows that set $U$ has the required
  property because the Hilbert polynomial is independent of the base
  field.
\end{proof}

\subsection{The covering scheme}
We use the results of the previous subsection to construct a
surjective map $W_0\to \CM$ from a scheme $W_0$.
\begin{thm}\label{thm:cover_W0}
  There exists a scheme $W_0$ of finite type over $\spec(\bbZ)$ such
  that $$W_0(S)=\left\{Z\subseteq \bbP^n_S\times_S
    \bbP^N_S\left\vert\begin{array}{l} (Z,\pr_1)\in \CM(S) \text{ and
        }\\ \func{\pr_2}{Z}{\bbP^N_S}\text{ is a closed
          immersion}\end{array}\right.\right\}$$ for all schemes $S$.

  The map $\func \tau {W_0} \CM$, given by mapping $Z\subseteq
  \bbP^n_S\times_S\bbP^N_S$ in $W_0(S)$ to $(Z,\pr_1)$ in $\CM(S)$ for
  all $S$, is surjective as a map of Zariski sheaves.
\end{thm}
\begin{proof}
  Let $H$ be the Hilbert scheme parameterizing closed subschemes of the
  product $\bbP^n\times\bbP^N$ with Hilbert polynomial
  $P(t):=p((m+1)t)$ with respect to the very ample sheaf $\calO(1,1):=
  \pr_1^*\calO_{\bbP^n}(1)\otimes \pr_2^*\calO_{\bbP^N}(1)$. Applying
  Propositions~\ref{prop:finite_closed_fiber},
  \ref{prop:CM+equidim_fibers2}, \ref{prop:isom_on_image_open} and
  \ref{prop:Hilbert_poly_loc_const} to the universal family
  \mbox{$\calZ_H\subseteq \bbP^n_H\times_H\bbP^N_H$} over $H$ and the
  projections $\calZ_H\to \bbP^n_H$ and $\calZ_H\to\bbP^N_H$, we see
  that all conditions in $W_0$ are open on $H$.

  For surjectivity of $\tau$, let $S$ be a scheme and $(C,i)\in
  \CM(S)$. By Theorem~\ref{thm:embedding_general_base}, we have an
  open covering $\{S_a\}$ of $S$ and closed immersions
  $\injfunc{j_a}{C_a}{\bbP^N_{S_a}}$ such that
  $j_a^*\calO_{\bbP^N_{S_a}}(1)=i^*\calO_{\bbP^n_{S_a}}(m)$. For every
  $a$, the induced closed immersion
  $\injfunc{(i_a,j_a)}{C_{S_a}}{\bbP^n_{S_a}\times_{S_a}\bbP^N_{S_a}}$
  satisfies that
  $(i_a,j_a)^*\calO(1,1)=i_a^*\calO_{\bbP^n_{S_a}}(m+1)$, and hence gives
  rise to an element of $W_0(S_a)$.
\end{proof}

\subsection{Refinement of the covering scheme}\label{sec:refinement}
\begin{lm}\label{lm:sheaf_isom_repr}
  Let $S$ be a locally Noetherian scheme, and let $\calE$ and $\calF$
  be finite locally free $\calO_S$-modules. Then the functor $\func F
  {({\mathbf{Sch}}_S)^\circ}{\mathbf{Sets}}$ defined by
  $$F(T)=\{\calO_T\text{-module isomorphisms }g^*\calE\ra
  g^*\calF\},$$ where $\func g T S$ denotes the structure map, is
  represented by a scheme of finite type over $S$.
\end{lm}
\begin{proof}
  We note first that $F$ is an open subfunctor of \mbox{$\func H
    {(\mathbf{Sch}_S)^\circ}{\mathbf{Sets}}$} defined by
  $H(T)=\Hom_{\calO_T}(g^*\calE,g^*\calF)$. Now the statement follows
  as $H$ is represented by the scheme
  $\mathbf{Spec}(\opn{Sym}(\calE\otimes \calF^\vee))$ that is of
  finite type over $S$.
\end{proof}

\begin{prop}\label{prop:cover_isom}
  Over the scheme $W_0$ of Theorem~\ref{thm:cover_W0}, consider the
  universal family $\calZ_0\subseteq
  \bbP^n_{W_0}\times_{W_0}\bbP^N_{W_0}$ with structure map $\func f
  {\calZ_0} W_0$.  Let $\calL:=\pr_1^*\calO_{\bbP^n_{W_0}}(m)$, where
  $m$ is as in Proposition~\ref{prop:point_kernel_m-reg}, and set
  $\calE:=f_*\calL$. Let $\func F
  {(\mathbf{Sch}_{W_0})^\circ}{\mathbf{Sets}}$ be the functor defined
  by $$F(T) = \{\calO_T\text{-module isomorphisms } \calO_T^{N+1} \ra
  g^*\calE\},$$ for every $W_0$-scheme $\func g T W_0$. Then $F$ is
  represented by a scheme $W_1$ that is of finite type over $W_0$.
\end{prop}
\begin{proof}
  We have that $(\calZ_0,\pr_1)\in \CM(W_0)$. Hence the sheaf $\calE$
  is locally free of rank $N+1$ by
  Proposition~\ref{prop:LocFree_BaseChange}. Then the statement is a
  special case of Lemma~\ref{lm:sheaf_isom_repr}.
\end{proof}

\begin{rem}\label{rem:notation_closed_imm}
  Let $T$ be a scheme. Then a morphism $\func h T W_1$ corresponds to
  a closed subscheme $Z\subseteq \bbP^n_T\times\bbP^N_T$ and an
  isomorphism
  $\iso{\sigma}{\calO_T^{N+1}}{f_*\pr_1^*\calO_{\bbP^n_T}(m)}$ of
  $\calO_T$-modules, where $\func f Z T$ denotes the structure map. By
  Proposition~\ref{prop:isom-embedding}, the isomorphism $\sigma$
  gives rise to a closed immersion
  $\injfunc{j_h:=j_\sigma}{Z}{\bbP^N_T}$.
\end{rem}
Recall that, by definition of $W_1$, the second projection
$\func{\pr_2}{Z}{\bbP^N_T}$ is a closed immersion. In the following,
we show that we can restrict ourselves to isomorphisms $\sigma$ such
that $j_\sigma=\pr_2$. Explicitly, we prove that there exists a closed
subscheme $W$ of $W_1$ such that a morphism $\func h T W_1$ factors
through $W$ if and only if $j_h=\pr_2$.
\begin{lm}\label{lm:morphism_equal}
  Let $\func {f,g} X Y$ be two morphisms over a scheme $S$, and
  suppose that $Y$ is separated over $S$. Let $X_0:=X\times_{Y\times_S
    Y} Y$ be the fiber product with respect to the morphism
  $\func{(f,g)}{X}{Y\times_S Y}$ and the diagonal map
  \mbox{$\func {\Delta:=(\id,\id)} Y Y\times_S Y$}, and denote the projections by
  $\func p {X_0} X$ and $\func {p'} {X_0} Y$. Then
  $f=g$ if and only if the first projection $p$ is an isomorphism.
\end{lm}
\begin{proof}
  Being the base change of the closed immersion $\Delta$, the morphism
  $p$ is a closed immersion.

  Suppose first that $f=g$. As then $\Delta\circ f=(f,f)=(f,g)\circ
  \id_X$, by the universal property of the fiber product, there exists
  a map $\func q X {X_0}$ such that $p\circ q=\id_X$, see the
  following commutative diagram \[\xymatrix{X \ar@/_1pc/[ddr]_{\id_X}
    \ar@/^1pc/[drr]^f \ar@{-->}[dr]^q & & \\ & X_0 \ar[r]^{p'}
    \ar[d]^{p} & Y\ar[d]^\Delta \\ & X \ar[r]^-{(f,g)} & Y\times_S
    Y\rlap{.} }\] It follows that the closed immersion $p$ is an
  isomorphism.

  Now assume conversely that $p$ is an isomorphism, and let $q$ be its
  inverse. Then we have $\opn{pr}_i\circ(f,g)\circ p\circ
  q=\opn{pr}_i\circ \Delta \circ p'\circ q$ for both projections
  $\func{\opn{pr}_1,\pr_2}{Y\times_S Y}Y$. In particular, we get that
  $f=p'\circ q$ and $g=p'\circ q$, and hence $f=g$.
\end{proof}

\begin{prop}\label{prop:morph_equal}
  Let $S$ be a Noetherian scheme, and let $\func {f,g}X Y$ be
  morphisms of $S$-schemes, where $X$ is projective and flat over $S$
  and $Y$ is separated over $S$. Suppose that all fibers of $X$ over
  $S$ have Hilbert polynomial $p(t)$ with respect to some very ample
  sheaf $\calL$ on $X$. Then there exists a closed subscheme $S'$ of
  $S$ such that a map $\func h T S$ factors through $S'$ if and only
  if the morphisms $\func{f_T}{X_T}{Y_T}$ and $\func{g_T}{X_T}{Y_T}$
  coincide.
\end{prop}
\begin{proof}
  Consider the fiber product $X_0$ of $\func{(f,g)}X {Y\times_S Y}$
  and the diagonal
  $\func \Delta Y{Y\times_S Y}$. Note that $\Delta$ is a
  closed immersion since $Y$ is separated over $S$, and therefore we
  can consider $X_0$ as a closed subscheme of $X$. The construction of
  $X_0$ commutes with base change and hence, by
  Lemma~\ref{lm:morphism_equal}, $f_T=g_T$ for a morphism $T\ra S$ if
  and only if the induced closed immersion $(X_0)_T\hookrightarrow
  X_T$ is an isomorphism.

  We consider the natural transformation $\calH ilb^p_{X_0/S}\ra \calH
  ilb^p_{X/S}$ given by considering a closed subscheme of $(X_0)_T$ as
  closed subscheme of $X_T$ via the inclusion $(X_0)_T\hookrightarrow
  X_T$. This natural transformation is a monomorphism. Note that since
  $X_0$ and $X$ are projective over $S$, the functors $\calH
  ilb^p_{X_0/S}$ and $\calH ilb^p_{X/S}$ are represented by schemes
  $\Hilb^p_{X_0/S}$ and $\Hilb^p_{X/S}$ that are projective over
  $S$. In particular, we have a commutative diagram
  $$\xymatrix{\Hilb^p_{X_0/S}\ar[r]^j\ar[dr]^-{\pi_0} &
    \Hilb^p_{X/S}\ar[d]^\pi \\ & S}$$ with projective, and in
  particular proper and separated, maps $\pi_0$ and $\pi$. Then it
  follows from \cite[Corollary 4.8]{Hartshorne} that the monomorphism
  $j$ is proper. By \cite[Corrolaire (18.12.6)]{EGAIV4}, this implies
  that $j$ is a closed immersion.

  The scheme $X$ as a closed subscheme of itself corresponds to an
  element of $\calH ilb^p_{X/S}(S)$, and hence to a map $S\to
  \Hilb^p_{X/S}$. Let $S'$ be the fiber product of $S$ and
  $\Hilb^p_{X_0/S}$ over $\Hilb^p_{X/S}$. Then $S'$ is a closed
  subscheme of $S$. Moreover, from the commutativity
  of the fiber diagram $$\xymatrix{S'\ar[r]\ar[d] & S \ar[d]\\
    \Hilb^p_{X_0/S} \ar[r] & \Hilb^p_{X/S}, }$$ we get that
  $(X_0)_{S'}= X_{S'}$.  Suppose conversely that $\func h T S$ is a
  morphism such that $(X_0)_T\hookrightarrow X_T$ is an isomorphism.
  Then the composition $T\ra S \ra \Hilb^p_{X/S}$ corresponds to the
  closed subscheme $X_T$ of $X_T$. Since $(X_0)_T=X_T$, the map $T\ra
  \Hilb^p_{X/S}$ factors through $\Hilb^p_{X_0/S}$. It follows that
  $T\ra S$ factors through the fiber product $S'$, and hence $S'$ has
  the required properties.
\end{proof}
\begin{prop}\label{prop:cover_W}
  Let $W_1$ be the $W_0$-scheme of Proposition~\ref{prop:cover_isom}.
  With the notation as in Remark~\ref{rem:notation_closed_imm}, there
  exists a closed subscheme $W$ of $W_1$ such that a morphism $\func
  h T{W_1}$ factors through $W$ if and only if $j_h=\pr_2$.
\end{prop}
\begin{proof}
  Over $W_1$ we have the universal family $\calZ_1\subseteq
  \bbP^n_{W_1}\times_{W_1}\bbP^N_{W_1}$ with induced closed immersion
  $\injfunc{j:=j_{\id}}{\calZ_1}{\bbP^N_{W_1}}$. As $j_h$ is the base
  change of $j$ by $h$, the statement follows directly from
  Proposition~\ref{prop:morph_equal}.
\end{proof}

Let $\func{\pi}{W}{\CM}$ be the composition $W\hookrightarrow W_1\ra
W_0\stackrel{\tau}\ra \CM$.
\begin{thm}\label{thm:pi_surj}
  The map $\func\pi W \CM$ is surjective as a map of Zariski
  sheaves. 

  In particular, the map
  $\func{\pi(\spec(k))}{W(\spec(k))}{\CM(\spec(k))}$ is surjective for
  every field $k$.
\end{thm}
\begin{proof}
  Let $S$ be a scheme, and let $(C,i)\in \CM(S)$ with structure map
  $\func f C S$. By Theorem~\ref{thm:embedding_general_base}, there
  exists an open cover $\{S_a\}$ of $S$, such that the sheaf
  $f_*i^*\calO_{\bbP^n_S}(m)$ is free of rank $N+1$ over every $S_a$,
  and a choice of basis of global sections gives rise to a closed
  immersion $\injfunc{j_a}{C_a}{\bbP^N_{S_a}}$ with
  $j_a^*\calO_{\bbP^N_{S_a}}(1)=i_a^*\calO_{\bbP^n_{S_a}}(m)$. Then
  the closed immersion $\func
  {(i_a,j_a)}{C_a}{\bbP^n_{S_a}\times_{S_a}\bbP^N_{S_a}}$ corresponds
  to an element of $W_0(S_a)$ mapping to $(C_a,i_a)\in \CM(S_a)$.  The
  choice of the basis gives a factorization of the corresponding map
  $S_a\to W_0$ through $W_1$. Moreover, by construction of $j_a$, it
  even factors through $W$.
\end{proof}
\begin{cor}
  For every scheme $S$ and map $S\to \CM$, the induced map
  $\func{\pi_S}{W\times_{\CM} S}S$ of schemes is surjective, that is,
  the map $\func\pi{W}\CM$ is surjective as representable map.
\end{cor}
\begin{proof}
  First we note that the statement makes sense since $\pi$ is
  representable by Corollary~\ref{cor:loc_fin_type}. 

  Let $S\to \CM$ be a map from a scheme $S$, and consider a map
  $\spec(k)\ra S$ for a field $k$.  Composition then gives an element
  of $\CM(\spec(k))$ having a lift in $W(\spec(k))$ by
  Theorem~\ref{thm:pi_surj}. Hence the original map $\spec(k)\to S$
  factors through the fiber product $W\times_{\CM}S$.  By
  \cite[Proposition 4.8]{GortzWed:AlgGeo}, this implies that $\pi_S$
  is a surjective map of schemes.
\end{proof}

\section{Smoothness of the covering}\label{chap:smooth}

Let $\func \pi {W}\CM$ be the surjective map of
Theorem~\ref{thm:pi_surj}. The goal of this section is to show that
$\pi$ is smooth. Moreover, we conclude that $\CM$ is an algebraic
space.

\begin{lm}\label{lm:lift_iso_nilpotent}
  Let $A$ be a ring, and let $I$ be a nilpotent ideal of $A$.  
  \begin{enumerate}
  \item\label{lm:lift_iso_nilpotent_iso} Let $\func \varphi M N$ be a
    homomorphism of $A$-modules, and let $\func{\bar
      \varphi}{M/IM}{N/IN}$ be the induced map. Suppose that $N$ is
    flat over $A$. Then $\varphi$ is an isomorphism if and only if
    $\bar \varphi$ is an isomorphism.
  \item\label{lm:lift_iso_nilpotent_free} Let $N$ be a flat
    $A$-module, and let $\func{\bar\varphi }{(A/I)^n}{N/IN}$ be an
    isomorphism. Then there exists an isomorphism
    $\func{\varphi}{A^n}{N}$ such that
    $\varphi\otimes\id_{A/I}=\bar\varphi$.
  \end{enumerate}
\end{lm}
\begin{proof}
  Clearly $\bar \varphi$ is an isomorphism if $\varphi$ is an
  isomorphism.

  Suppose conversely that $\bar \varphi$ is an isomorphism, and let
  $C$ be the cokernel of $\varphi$. Then we have that
  $C/IC=\coker(\bar \varphi)=0$. It follows that $C=IC=I^2C=\ldots =0$
  since $I$ is nilpotent. This shows that $\varphi$ is surjective.

    Now let $K:=\ker(\varphi)$, and consider the short exact
    sequence $$\xymatrix{0\ar[r] & K \ar[r] & M \ar[r]^\varphi &
      N\ar[r] & 0.}$$ Since $N$ is flat, the sequence remains exact after
    tensoring with $A/I$. In particular, we get that $K/IK=\ker(\bar
    \varphi)=0$. As for the cokernel, this implies that $K=0$ and
    assertion~\ref{lm:lift_iso_nilpotent_iso} follows.

    To show assertion~\ref{lm:lift_iso_nilpotent_free}, we lift the
    basis of $N/IN$ given by $\bar \varphi$ to $N$, and get a
    well-defined map \mbox{$\func \varphi {A^n} N$} that is an
    isomorphism after tensoring with $A/I$.  Now it follows from
    assertion~\ref{lm:lift_iso_nilpotent_iso} that $\varphi$ is an
    isomorphism.
\end{proof}

\begin{thm}\label{thm:form_smooth}
  The map $\func \pi{W}\CM$ is formally smooth, that is, for every
  ring $A$ with nilpotent ideal $I$ and all morphism $\spec(A/I)\to W$
  and $\spec(A)\to \CM$ such that the 
  diagram
  \begin{align}\label{eqn:form_smooth}
    \begin{array}{rcl}
      \xymatrix{\spec(A/I) \ar@{^(->}[d] \ar[r] & W\ar[d]^\pi \\
        \spec(A)\ar@{-->}[ur] \ar[r] & \CM}
    \end{array}
  \end{align}
  commutes, there exists at least one morphism $\spec(A)\dashrightarrow
  W$ making the diagram commute.
\end{thm}
\begin{proof}
  Consider a commutative diagram as in the statement of the
  theorem. The map $\spec(A)\ra \CM$ corresponds to an element $(C,i)$
  in $\CM(\spec(A))$ with structure morphism $\func f C \spec(A)$.
  Base change by the closed immersion $\spec(A/I)\hookrightarrow
  \spec(A)$ gives rise to an element $(\bar C,\bar i)\in
  \CM(\spec(A/I))$ with structure morphism $\func{\bar f}{\bar
    C}\spec(A/I)$, where we write $\bar C:=C_{A/I}$.

  The map $\spec(A/I)\to W$ corresponds to a closed subscheme $\bar Z$
  of the product $\bbP^n_{A/I}\times_{A/I}\bbP^N_{A/I}$ with structure
  map $\func{\bar g}{\bar Z}{\spec(A/I)}$ and an isomorphism
  $\iso{\bar \sigma}{\calO_{\spec(A/I)}^{N+1}}{\bar
    g_*\pr_1^*\calO_{\bbP^n_{A/I}}}(m)$ such that $j_{\bar
    \sigma}=\pr_2$. 

  Note that $(\bar Z,\bar\pr_1)=(\bar C,\bar i)$ in $\CM(\spec(A/I))$
  since diagram~(\ref{eqn:form_smooth}) commutes. Hence there is an
  isomorphism $\iso\alpha {\bar C}{\bar Z}$ with $\pr_1\circ
  \alpha=\bar i$. Composing the induced isomorphism $\bar
  g_*\pr_1^*\calO_{\bbP^n_{A/I}}(m) \cong \bar f_*\bar i^*
  \calO_{\bbP^n_{A/I}}(m)$ with $\bar \sigma$ gives an isomorphism
  $\iso{\bar\rho} {\calO_{\spec(A/I)}^{N+1}} {\bar f_*\bar
    i^*\calO_{\bbP^n_{A/I}}(m)}$. Then the closed immersion
  $j_{\bar\rho}$ associated to $\bar \rho$ equals the composed map $\injfunc{\pr_2\circ
    \alpha}{\bar C}{\bbP^N_{A/I}}$.

  Using the identification $(f_*i^*\calO_{\bbP^n_A}(m))_{A/I}\cong
  \bar f_* \bar i^*\calO_{\bbP^n_{A/I}}(m)$ of
  Proposition~\ref{prop:LocFree_BaseChange}, we can lift $\bar \rho$
  to an isomorphism
  $\iso\rho{\calO_A^{N+1}}{f_*i^*\calO_{\bbP^n_A}}(m)$ by
  Lemma~\ref{lm:lift_iso_nilpotent}\ref{lm:lift_iso_nilpotent_free}. By
  Proposition~\ref{prop:isom-embedding}, we get an induced closed
  immersion $\injfunc{j_\rho}{C}{\bbP^N_A}$. 

  The combined closed immersion
  $\injfunc{(i,j_\rho)}{C}{\bbP^n_A\times_A\bbP^N_A}$ gives rise to a
  map $\func\gamma{\spec(A)}W$.  By construction, the base change of
  $j_\rho$ by the closed immersion \mbox{$\spec(A/I)\hookrightarrow
    \spec(A)$} is the map $\pr_2\circ\alpha=j_{\bar \rho}$. It follows
  that the constructed map $\gamma$ makes the
  diagram~\eqref{eqn:form_smooth} commute.
\end{proof}
\begin{rem}
  Note that the map $\gamma$ constructed in the proof is not
  unique. It depends on the choice of a lift $\rho$ of $\bar \rho$. In
  particular, the map $\func\pi W \CM$ is not formally
  \'etale. However, the existence of a smooth cover suffices to show
  that $\CM$ is an algebraic space.
\end{rem}

\begin{prop}\label{prop:alg_space_smooth}
  Let $\func{\calA}{\mathbf{Sch}^\circ}{\mathbf{Sets}}$ be a functor
  that is a sheaf in the \'etale topology. Suppose that there exists a
  scheme $X$ with a representable, surjective and smooth map $X\to
  \calA$. Then $\calA$ is an algebraic space.
\end{prop}
\begin{proof}
  See \cite[Chapter 8]{LMB:Champs} for the construction of an \'etale
  cover of $\calA$.
\end{proof}
\begin{thm}
  The functor $\CM$ is an algebraic space of finite type over $\spec(\bbZ)$.
\end{thm}
\begin{proof}
  We have seen in Theorem~\ref{thm:sheaf} that $\CM$ is a sheaf in the
  \'etale topology. Moreover, we have a scheme $W$ of finite type
  over $\spec(\bbZ)$ and a map $\func\pi{W}\CM$ that is
  representable by Proposition~\ref{prop:representable}, surjective by
  Theorem~\ref{thm:pi_surj}, locally of finite presentation by
  Corollary~\ref{cor:loc_fin_type} and formally smooth by
  Theorem~\ref{thm:form_smooth}. Now it follows from
  Proposition~\ref{prop:alg_space_smooth} that $\CM$ is an algebraic
  space.
\end{proof}

\section{\texorpdfstring{Properness of $\CM$}{Properness of CM}}
We conclude by showing that $\CM$ satisfies the valuative criterion
for properness.
\subsection{Flat families over a discrete valuation ring} Flatness over
a discrete valuation ring can be described in terms of associated
points.
\begin{lm}[{\cite[Proposition 6 in Lecture 6]{Mumford:Lectures}}]\label{lm:flat_over_DVR}
  Let $R$ be a discrete valuation ring, and let $X$ be a scheme over
  $\spec(R)$ with structure morphism $\func f X \spec(R)$. A coherent
  sheaf $\calF$ on $X$ is flat over $\spec(R)$ if and only if every
  associated point of $\calF$ is contained in the generic fiber
  $X_\eta$, where $\eta$ denotes the generic point of
  $\spec(R)$. 
\end{lm}

\begin{prop}\label{prop:dim_fibers}
  Let $R$ be a discrete valuation ring, and let $\eta$ be the generic
  point of $\spec(R)$. Let $\func f X \spec(R)$ be a flat and proper
  morphism, and suppose that the generic fiber $X_\eta$ has pure
  dimension~$n$. Then $X$ has pure dimension $n+1$. Moreover, the
  closed fiber has pure dimension $n$.
\end{prop}
\begin{proof}
  Let $X'$ be an irreducible component of $X$. Since $f$ is flat and
  proper, the generic point is mapped to $\eta$ and the closed points
  in $X'$ are mapped to the closed point in $\spec(R)$. Hence the
  restriction of $f$ to $X'$ is surjective, and it follows from
  \cite[Proposition~14.107]{GortzWed:AlgGeo} that
  $\dim(X')=\dim(\spec(R))+\dim((X')_\eta)=1+n$.

  By \cite[Corollaire (14.2.4)]{EGAIV3}, the set of dimensions of the
  irreducible components of the closed fiber is contained in the set
  of dimensions of the irreducible components of the generic fiber.
  This implies the statement on the closed fiber.
\end{proof}
\begin{cor}\label{cor:codim_point}
  Let $R$ be a discrete valuation ring, and let $X$ be a scheme that
  is flat and proper over $\spec(R)$. Suppose that the generic fiber
  $X_\eta$ has pure dimension $1$. Let $x$ be a closed point of
  $X$. Then $x$ has codimension $2$ in every irreducible component of
  $X$ containing it. In particular, we have $\dim(\calO_{X,x})=2$.
\end{cor}
\begin{proof}
  By Proposition~\ref{prop:dim_fibers}, the scheme $X$ has pure
  dimension $2$. Let $x\in X$ be a closed point, and let $X'$ be an
  irreducible component of $X$ containing it. Note that $X'$ is flat
  and proper over $\spec(R)$. In particular, the closed fiber $(X')_0$
  has pure dimension $1$ by Proposition~\ref{prop:dim_fibers}, it
  contains the closed point $x$, and we have
  $\codim(\{x\},(X')_0)=1$. Moreover, the closed fiber $(X')_0$ has
  codimension $1$ in $X'$ by {\em Krull's Haupt\-ideal\-satz}. This
  shows that $x$ has codimension at least $2$ in $X'$.
\end{proof}

\subsection{Family of Cohen--Macaulay curves over a discrete valuation ring}

We fix the following notation. Let $R$ be a discrete valuation ring
with maximal ideal $\m=(\pi)$ and field of fractions $K=\Quot(R)$.

For a family of Cohen--Macaulay curves $(C,i)\in \CM(\spec(R))$, we
denote the scheme-theoretic image $i(C)$ of $i$ by $Z$. Then the
morphism $i$ factors as \[\xymatrix{C \ar[rr]^i \ar[dr]_h && \bbP^n_R
  \\ &Z \ar@{^(->}[ur] & }.\] By construction, the associated map
$\func{h^\#}{\calO_Z}{h_*\calO_C}$ is injective, and we set
$Y:=\supp(\coker(h^\#))$ for the locus of points in $Z$ where $C$ and
$Z$ are not isomorphic.

Further we write $(C_K,i_K)$ in $\CM(\spec(K))$ for the restriction of
$(C,i)$ to the generic point in $\spec(R)$. Then the generic fiber
$Z_K$ of $Z$ is the scheme-theoretic image $i_K(C_K)$ of $C_K$ in
$\bbP^n_K$. In particular, with the induced map $\func {h_K}{C_K}Z_K$,
we have that $\func{h_K^\#}{\calO_{Z_K}}{(h_K)_*\calO_{C_K}}$ is
injective, and we set $Y_K:=\supp(\coker(h_K^\#))$. Note that $Y_K$ is
the restriction of $Y$ to the generic fiber.

\subsubsection{Scheme-theoretic image}
First we study the scheme-theoretic image $Z$. We show that it is the
scheme-theoretic closure of its generic fiber $Z_K$ in $\bbP^n_R$, and
that it has pure dimension $2$.
\begin{prop}\label{prop:image_description}
  The scheme-theoretic image $Z$ of $C$ in $\bbP^n_R$ is the
  scheme-theoretic closure of $Z_K$ in $\bbP^n_R$. In particular, $Z$
  is flat over $\spec(R)$, it has pure dimension $2$ and
  $\dim(\calO_{Z,z})=2$ for every closed point $z$ in $Z$.
\end{prop}
\begin{proof}
  The scheme $C$ is flat over $\spec(R)$, and hence its generic fiber
  $C_K$ is scheme-theoretically dense in $C$. Thus the scheme $Z$ is
  the scheme-theoretic image of the composition $C_K\hookrightarrow C
  \ra \bbP^n_R$. Moreover, by definition of $Z_K$, the
  scheme-theoretic closure of $Z_K$ in $Z$ equals the scheme-theoretic
  image of the composition $C_K\to Z_K \hookrightarrow \bbP^n_K \to
  \bbP^n_R$. Now the first statement follows as the
  diagram $$\xymatrix{C_K \ar[r]^{i_K} \ar[d] & \bbP^n_K \ar[d] \\ C
    \ar[r]^i & \bbP^n_R }$$ commutes. Moreover, the associated points
  of $Z$ lie in the generic fiber $Z_K$, and hence $Z$ is flat over
  $\spec(R)$ by Lemma~\ref{lm:flat_over_DVR}.

  Note that $Z_K$ has pure dimension $1$ by
  Proposition~\ref{prop:image}. Then the remaining statements follow
  directly from Propostion~\ref{prop:dim_fibers} and
  Corollary~\ref{cor:codim_point}.
\end{proof}

\begin{lm}\label{lm:point_closed}
  Let $\func f X Y$ be a finite morphism of locally Noetherian
  schemes. A point $x\in X$ is closed if and only if its image $f(x)$
  is a closed point of $Y$.
\end{lm}
\begin{proof}
  Since the finite morphism $f$ is closed, the image of a closed point
  is closed. Suppose conversely that the image $f(x)$ is a closed
  point. By \cite[Corollaire (6.1.7)]{EGAII}, the morphism $f$ is
  quasi-finite. In particular, the fiber $f^{-1}(f(x))$ is finite and
  closed, that is, it consists of finitely many closed points.
\end{proof}

\begin{prop}\label{prop:codim_point_image}
  Let $(C,i)\in\CM(\spec(R))$ for a discrete valuation ring $R$, and
  let $\func h C Z$ be the induced map to the image. Then
  $\dim(\calO_{C,x})=\dim(\calO_{Z,h(x)})$ for every point $x\in C$.
\end{prop}
\begin{proof}
  After passing to suitable affine neighborhoods of $x$ and $h(x)$,
  the {\em Cohen--Seidenberg Theorem} \cite[(5.E) Theorem
  5]{Mats:algebra} gives that $\dim(\calO_{Z,h(x)})\geq
  \dim(\calO_{C,x})$ for every $x\in C$. We have to show that the
  inequality cannot be strict.

  Since both $C$ and $Z$ have pure dimension $2$,
  Lemma~\ref{lm:point_closed} implies that $\dim(\calO_{Z,h(x)})=2$ if
  and only if $\dim(\calO_{C,x})=2$. It remains to eliminate the
  possibility that $\dim(\calO_{Z,h(x)})=1$ and
  $\dim(\calO_{C,x})=0$. If $x$ is a generic point of $C$, then it
  lies in the generic fiber $C_K$. Consequently, the image $h(x)$ lies
  in the generic fiber $Z_K$ that, by
  Corollary~\ref{cor:scheme-theor-image_CM}, has pure
  dimension~$1$. It follows that \mbox{$\dim(\calO_{Z,h(x)})=
    \dim(\calO_{Z_K,h(x)})\leq 1$}. As the point $h(x)$ is not closed
  in $Z_K$ by Lemma~\ref{lm:point_closed}, we conclude that
  $\dim(\calO_{Z,h(x)})=0$.
\end{proof}

\begin{cor}\label{cor:dim_local_ring_image}
  Let $(C,i)\in \CM(\spec(R))$ be a family of Cohen--Macaulay curves
  over a discrete valuation ring $R$. Let $x_1,x_2\in C$ be two points
  in $C$ such that $h(x_1)=h(x_2)$, where $\func h C Z$ is the induced
  map onto the image. Then $\dim(\calO_{C,x_1})=\dim(\calO_{C,x_2})$.
\end{cor}
\begin{proof}
  This follows directly from Proposition~\ref{prop:codim_point_image}.
\end{proof}

\subsubsection{Cohen--Macaulayness of the direct image }
Next, we show that the direct image sheaf $h_*\calO_{C}$ is
Cohen--Macaulay as a $\calO_Z$-module.

\begin{lm}[{\cite[Corollaire (5.7.11)]{EGAIV2}}]\label{lm:CM_base_change}
  Let $\func \varphi A B$ be a homomorphism of Noetherian rings, and
  let $M$ be a module over $B$ that is finitely generated as an
  $A$-module.
  \begin{enumerate}
  \item\label{lm:CM_base_change1} Suppose that $M$ is Cohen--Macaulay
    as an $A$-module. Then $M$ is Cohen--Macaulay as a $B$-module.
  \item\label{lm:CM_base_change2} Suppose that
    $\dim_{B_{\q_1}}(M_{\q_1})=\dim_{B_{\q_2}}(M_{\q_2})$ for all
    prime ideals $\q_1,\q_2\in \spec(B)$ such that
    $\varphi^{-1}(\q_1)=\varphi^{-1}(\q_2)$. If $M$ is Cohen--Macaulay
    as a $B$-module, then $M$ is Cohen--Macaulay as a module over $A$.
  \end{enumerate}
\end{lm}

\begin{prop}\label{prop:DVR_dir_image_CM}
  The direct image sheaf $h_*\calO_C$ is Cohen--Macaulay as an
  $\calO_Z$-module.
\end{prop}
\begin{proof}
  Note that since both $\spec(R)$ and all fibers of $C$ over
  $\spec(R)$ are Cohen--Macaulay, it follows from
  Proposition~\ref{prop:fiber_CM} that $C$ is Cohen--Macaulay. In
  particular, the sheaf $\calO_C$ is Cohen--Macaulay as an
  $\calO_C$-module. 

  By Lemma~\ref{lm:CM_base_change}\ref{lm:CM_base_change2}, it
  suffices to show that $\dim(\calO_{C,x_1})=\dim(\calO_{C,x_2})$ for
  all points $x_1,x_2\in C$ such that $h(x_1)=h(x_2)$. Then
  Corollary~\ref{cor:dim_local_ring_image} concludes the proof.
\end{proof}

\subsubsection{The non-isomorphism locus}\label{subsec:non-isom-locus}

In the next step, we show that the locus where $C$ and $Z$ are not
isomorphic is the union of the non-Cohen--Macaulay locus in $Z$ and
the closure of the locus where the generic fibers $C_K$ and $Z_K$ are
not isomorphic.

Recall that we defined $Y=\supp(\coker(h^\#))$, for the natural
inclusion $\injfunc{h^\#}{\calO_{Z}}{h_*\calO_C}$. Then
$Y_K=\supp(\coker((h_K)^\#))$ is the intersection of $Y$ with the
generic fiber $Z_K$. Let $Y_1$ be the closure of the set $Y_K$ in
$Z$.

\begin{lm}\label{lm:non-CM_closed}
  The set $Y_2:=\{z\in Z\mid \calO_{Z_,z} \text{ is not
    Cohen--Macaulay}\}$ is closed in $Z$.
\end{lm}
\begin{proof}
  As $Z$ is a closed subscheme of the regular scheme $\bbP^n_R$, the
  statement follows directly from \cite[Corollaire (6.11.3)]{EGAIV2}.
\end{proof}

In Proposition~\ref{prop:non-isom-locus}, we show that these two sets
$Y_1$ and $Y_2$ constitute the entire non-isomorphism locus. In particular, this
implies that $Y$ is uniquely determined by the generic fiber
$(C_K,i_K)$.

The following results will be useful, see also
\cite[Lemma~36]{Kollar:HH}. We use the notation $\depth_Y(\calF) :=
\inf_{y\in Y} \depth_{\calO_{X,y}}(\calF_y)$ for a closed subset $Y$
of a locally Noetherian scheme $X$ and a coherent
$\calO_X$-module~$\calF$.
\begin{prop}\label{prop:isom_on_open_depth}
  Let $X$ be a locally Noetherian scheme, and let $\func \alpha \calF
  \calG$ be a morphism of $\calO_X$-modules where $\calF$ is coherent
  and $\calG$ is quasi-coherent.
  \begin{enumerate}
  \item\label{item:isom_on_open_depth1} Suppose that there exists an
    open subscheme $U$ of $X$ such that the restriction of $\alpha$ to
    $U$ is injective. If $\depth_{X\setminus U}(\calF)\geq 1$, then
    $\alpha$ is injective.
  \item\label{item:isom_on_open_depth2} Suppose that there exists an
    open subscheme $U$ of $X$ such that the restriction of $\alpha$ to
    $U$ is an isomorphism. If $\Ass(\calG)\subseteq U$ and
    $\depth_{X\setminus U}(\calF)\geq 2$, then $\alpha$ is an
    isomorphism.
  \end{enumerate}
\end{prop}
\begin{proof}
  Since all properties are local on $X$, we can without loss of
  generality assume that $X=\spec(A)$ for a Noetherian ring $A$. Then
  we have $X\setminus U=V(I)$ for an ideal $I$ of $A$, and
  $\calF=\tilde M$ and $\calG=\tilde N$ for $A$-modules $M$ and
  $N$. Note that
  $$\depth_{X\setminus U}(\calF) = \inf_{\p\supseteq
    I}\depth(M_\p)=\operatorname{grade}(I,M)$$ by \cite[Proposition
  1.2.10]{Bruns-Herzog}.
  
  Suppose first that $\alpha\vert_U$ is injective and
  $\depth_{X\setminus U}(\calF)\geq 1$. Then there exists an element
  $a\in I$ that is not a zero divisor of $M$. Let $K:=\ker(\alpha)
  \subset M$, and take $m \in K$. Since the open subset $D(a)$ is
  contained in $U$, the restriction of $\alpha$ to $D(a)$ is
  injective, and hence $K_a=0$. Therefore there exists some $n\in\bbN$
  such that $a^n m=0$. But $a$ is not a zero divisor of $M$, so $m=0$
  and $\alpha$ is injective.

  To show that the conditions in
  assertion~\ref{item:isom_on_open_depth2} moreover imply that
  $\alpha$ is surjective, suppose that $\alpha(M)\subsetneq N$. Let
  $(a,b)$ be a $M$-regular sequence in $I$. In particular, the element
  $a$ is not a zero divisor of $M$. Moreover, since
  $\Ass(N)=\Ass(M)\cap U$, it follows that $a$ is not a zero divisor
  of $N$. The restriction of $\alpha$ to both open subsets $D(a),
  D(b)\subseteq U$ is surjective, so there exists an element $n\in
  N\setminus \alpha(M)$ such that $an,bn\in \alpha(M)$, say
  $an=\alpha(m_1)$ and $bn=\alpha(m_2)$ for $m_1,m_2\in M$. Since
  $\alpha$ is injective, it follows that $bm_1=am_2$. In particular,
  the element $b m_1$ lies in $aM$, and thus, as the sequence $(a,b)$
  is $M$-regular, there exists an element $m_1'\in M$ such that
  $m_1=am_1'$. We get that $an = \alpha(m_1) = a\alpha(m_1')$. Since
  $a$ is not a zero divisor in $N$, this implies that
  $n=\alpha(m_1')\in \alpha(M)$, a contradiction. \qedhere
\end{proof}
\begin{cor}\label{cor:isom-open_depth}
  Let $X$ be a locally Noetherian scheme, and let $\calF$ be a
  coherent sheaf on $X$. Let $\injfunc j U X$ be an open subscheme of
  $X$ such that $\depth_{X\setminus U}(\calF)\geq 2$. Then $\calF\cong
  j_*(\calF\vert_U)$.
\end{cor}
\begin{proof}
  Consider the natural map $\func \alpha \calF{j_*(\calF\vert_U)}$
  that is the identity when restricted to $U$. Since
  $\Ass(j_*(\calF\vert_U))=\Ass(\calF\vert_U)\subseteq U$ by
  \cite[Proposition (3.1.13)]{EGAIV2}, it follows from
  Proposition~\ref{prop:isom_on_open_depth}\ref{item:isom_on_open_depth2}
  that $\alpha$ is an isomorphism.
\end{proof}

\begin{prop}\label{prop:non-isom-locus}
  In the situation of Subsection~\ref{subsec:non-isom-locus}, the
  non-i\-so\-mor\-phism locus is the union of the closure $Y_1$ of the
  non-isomorphism locus of the generic fiber and the
  non-Cohen--Macaulay locus $Y_2$, that is, $Y=Y_1\cup Y_2$.
\end{prop}
\begin{proof}
  The non-isomorphism locus $Y_K$ of the generic fiber is contained in
  the closed set $Y$ and hence even its closure $Y_1$ in
  $Z$. Moreover, since the sheaf $h_*\calO_C$ is Cohen--Macaulay as an
  $\calO_Z$-module by Proposition~\ref{prop:DVR_dir_image_CM}, the
  non-Cohen--Macaulay locus $Y_2$ of $Z$ is contained in $Y$. This
  shows the inclusion $Y\supseteq Y_1\cup Y_2$.

  Suppose that the inclusion is strict, that is, $Y':=Y_1\cup
  Y_2\subsetneq Y$. Since $Y\cap Z_K=Y_K=Y'\cap Z_K$, the complement
  $Y\setminus Y'$ lies over the closed point of $\spec(R)$. As the
  closed fiber of $Y$ has dimension 0 by Lemma~\ref{lm:iso_alt}, it
  follows that $Y\setminus Y'$ consists of closed points that have
  codimension $2$ in $Z$ by
  Proposition~\ref{prop:image_description}. Consider the scheme
  $X:=Z\setminus Y'$ and its open subscheme $U:=Z\setminus Y$. Let
  further $\injfunc \alpha {\calO_X}{(h_*\calO_C)\vert_X}$ be the
  restriction of $h^\#$ to $X$. By definition of $U$, the restriction
  of $\alpha$ to $U$ is an isomorphism. The coherent sheaf
  $(h_*\calO_C)\vert_X$ is Cohen--Macaulay by
  Proposition~\ref{prop:DVR_dir_image_CM}, and hence
  $\depth_{X\setminus U}((h_*\calO_C)\vert_X)=2$. Finally we observe
  that all associated points of $\calO_Z$ lie in the generic fiber,
  and in particular in $U$. Then it follows from
  Proposition~\ref{prop:isom_on_open_depth}\ref{item:isom_on_open_depth2}
  that $\alpha$ is an isomorphism, contradicting the assumption. We
  conclude that $Y=Y_1\cup Y_2$ as claimed.
\end{proof}

\begin{prop} \label{prop:C_determined} Consider the open subscheme
  $U:=Z_K\cup (Z\setminus Y)$ of $Z$, and let $\injfunc j U Z$ be the
  open immersion. The $\calO_Z$-algebras $h_*\calO_C$ and
  $j_*((h_*\calO_C)\vert_U)$ are isomorphic.
\end{prop}
\begin{proof}
  The complement $Z\setminus U=Y\setminus Y_K$ is contained in the
  zero-dimensional closed fiber of $Y$. In particular, every point in
  $Z\setminus U$ has codimension $2$ in $Z$ by
  Proposition~\ref{prop:image_description}. As $h_*\calO_C$ is
  Cohen--Macaulay by Proposition~\ref{prop:DVR_dir_image_CM}, we
  consequently have that $\depth_{Z\setminus U}(h_*\calO_C)=2$.  Then
  $h_*\calO_C\cong j_*((h_*\calO_C)\vert_U)$ by
  Corollary~\ref{cor:isom-open_depth}. The isomorphism there is an
  isomorphism of $\calO_Z$-algebras.
\end{proof}

\subsection{Valuative criterion for separatedness}
The results of the previous section particularly imply that every
$(C,i) \in \CM(\spec(R))$ is uniquely determined by its generic fiber
$(C_K,i_K)$, that is, the functor $\CM$ satisfies the valuative
criterion for separatedness.
\begin{thm}[Valuative criterion for separatedness]\label{thm:valu-crit-separ}
  Let $R$ be a discrete valuation ring with field of fractions
  $K$. For every commutative diagram \[\xymatrix{\spec(K)\ar[r] \ar[d]
    & \CM \ar[d] \\ \spec(R) \ar[r] \ar@{-->}[ur] & \spec(\bbZ)}\]
  there exists at most one map $\spec(R)\dashrightarrow \CM$ making
  the diagram commute. In other words, every element $(C,i)\in
  \CM(\spec(R))$ is uniquely determined by its generic fiber
  $(C_K,i_K)\in \CM(\spec(K))$.
\end{thm}
\begin{proof}
  Let $(C,i)$ and $(C',i')$ be two elements in $\CM(\spec(R))$ having
  the same generic fiber $(C_{K},i_{K})=(C_{K}',i_{K}')$ in
  $\CM(\spec(K))$. Then there exists an isomorphism $\iso
  {\alpha_K}{C_{K}}{C_{K}'}$ with $i_{K}'\circ {\alpha_K}=i_{K}$. In
  particular, $C_K$ and $C_{K}'$ have the same scheme-theoretic image
  $Z_K$ in $\bbP^n_K$ and the same non-isomorphism locus $Y_K\subset
  Z_K$. By Proposition~\ref{prop:image_description}, it follows that
  $C$ and $C'$ have the same scheme-theoretic image $Z$ in $\bbP^n_R$,
  and we write $\func h C Z$ and $\func {h'}{C'}Z$ for the
  restrictions of $i$ and $i'$ to the image. Then the isomorphism
  $\alpha_K$ induces an isomorphism $\iso{\beta_{Z_K}}
  {(h_*\calO_{C})\vert_{Z_K} }{(h'_*\calO_{C'})\vert_{Z_K}}$ of
  $\calO_{Z_K}$-algebras. By Proposition~\ref{prop:non-isom-locus},
  both non-isomorphism loci $Y=\supp(\coker(h^\#))$ and
  $Y'=\supp(\coker((h')^\#))$ are given as the union of closure of the
  non-isomorphism locus $Y_K$ of the generic fiber and
  non-Cohen--Macaulay locus of $Z$, and we have $Y=Y'$. Moreover, we
  get an induced $\calO_{Z\setminus Y}$-algebra isomorphism
  $\iso{\beta_{Z\setminus Y}}{(h_*\calO_{C})\vert_{Z\setminus
      Y}}{(h'_*\calO_{C'})\vert_{Z\setminus Y}}$.
 
  Note that the homomorphisms $\beta_{Z_K}$ and $\beta_{Z\setminus Y}$
  coincide on the intersection $ Z_K\cap (Z\setminus Y) = Z_K\setminus
  Y_K$, and hence they glue to an isomorphism
  $\iso{\beta}{(h_*\calO_{C})\vert_{U}} {(h'_*\calO_{C'})\vert_{U}}$
  on the union $U:= Z_K \cup (Z\setminus Y)$.

  Let $\injfunc j U Z$ be the open immersion. Then we get an
  isomorphism $\iso{j_*\beta} {j_*((h_*\calO_{C})\vert_U)}
  {j_*((h'_*\calO_{C'})\vert_U)}$ of $\calO_Z$-algebras. By
  Proposition~\ref{prop:C_determined}, this gives an isomorphism
  $h_*\calO_{C}\stackrel{\sim}{\longrightarrow}h'_*\calO_{C'}$ of
  $\calO_Z$-algebras, that is, an isomorphism $\iso\alpha{C}{C'}$ over
  $Z$, and hence over $\bbP^n_R$. This shows that $(C,i)=(C',i')$ in
  $\CM(\spec(R))$.
\end{proof}

\subsection{Valuative criterion for properness}
 
The aim of this subsection is to show that any element in
$\CM(\spec(K))$ can be lifted to $\CM(\spec(R))$.

So take $(C_K,i_K)\in \CM(\spec(K))$. Let $Z_K\subseteq \bbP^n_K$ be
the scheme-theoretic image of $i_K$, and $\func {h_K}{C_K}{Z_K}$ be
the induced map to it. Recall that $Z_K$ has pure dimension $1$ by
Corollary~\ref{cor:scheme-theor-image_CM}.

Let $Z$ be the scheme-theoretic closure of $Z_K$ in $\bbP^n_R$. Then
$Z$ is flat by Lemma~\ref{lm:flat_over_DVR}, it has pure dimension $2$
by Proposition~\ref{prop:dim_fibers} and all closed points have
codimension $2$ by Corollary~\ref{cor:codim_point}. Let $Y_1$ be the
closure of the non-isomorphism locus $Y_K:=\supp(\coker(\calO_{Z_K}\to
(h_K)_*\calO_{C_K}))$ in $Z$ and $Y_2:= \{z\in Z\mid \calO_{Z,z}
\text{ is not Cohen--Macaulay}\}$ be the non-Cohen--Macaulay locus of
$Z$, and set $Y:=Y_1\cup Y_2$. Note that $Y_2$, and hence also $Y$, is
closed in $Z$ by Lemma~\ref{lm:non-CM_closed}.

\begin{prop}\label{prop:codim_2}
  Consider the open subscheme $U:=Z_K\cup (Z\setminus Y)$ of $Z$. Then
  $\codim(Z\setminus U,Z)=2$.
\end{prop}
\begin{proof}
  Let $z\in Z\setminus U$. We have to show that
  $\dim(\calO_{Z,z})=2$. Note that $Z\setminus U=(Y_1\setminus
  Y_K)\cup Y_2$. Suppose first that $z\in Y_1\setminus Y_K$. Then
  $z\in\overline{\{y\}}$ for a point $y\in Y_K$, and hence
  $\dim(\calO_{Z,z})>\dim(\calO_{Z,y})$. Moreover, we have
  $\dim(\calO_{Z,y})=\dim(\calO_{Z_K,y})>0$ since the point $y$ is
  closed in the scheme $Z_K$ of pure dimension $1$. Next suppose that
  $z\in Y_2$. Then $z$ lies in the closed fiber over $\spec(R)$ and,
  since $Z$ is flat over $\spec(R)$, the local ring $\calO_{Z,z}$ is a
  flat $R$-algebra. In particular, the generator $\pi$ of the maximal
  ideal of $R$ is not a zero divisor of $\calO_{Z,z}$ and hence
  $\depth_{\calO_{Z,z}}(\calO_{Z,z})\geq 1$. But we have
  $\dim(\calO_{Z,z})>\depth(\calO_{Z,z})$ since $\calO_Z$ is not
  Cohen--Macaulay at the point $z$.
\end{proof}

On $U_1=Z_K$ we have the coherent $\calO_{U_1}$-algebra
$\calA_1:=(h_K)_*\calO_{C_K}$, and on $U_2=Z\setminus Y$ we consider
the coherent $\calO_{U_2}$-algebra $\calA_2:=\calO_Z\vert_{U_2}$. Note
that the intersection $U_1\cap U_2=Z_K\setminus Y_K$ is the locus
where $\calA_1$ and $\calA_2$ coincide, and hence they glue to a
coherent $\calO_U$-algebra $\calA_U$ on the union $U=U_1\cup U_2$.
\begin{lm}\label{lm:A_U_is_CM}
  The $\calO_U$-module $\calA_U$ is Cohen--Macaulay.
\end{lm}
\begin{proof}
  To show that the restriction $\calA_1=(h_K)_*\calO_{C_K}$ of
  $\calA_U$ to $U_1$ is Cohen--Macaulay we can make use of
  Lemma~\ref{lm:CM_base_change}. Then it suffices to show that
  $\dim(\calO_{Z_K,x_1})=\dim(\calO_{Z_K,x_2})$ for all $x_1,x_2\in
  Z_K$ with $h_K(x_1)=h_K(x_2)$. As $Z_K$ has pure dimension $1$, this
  follows from Lemma~\ref{lm:point_closed}.

  The restriction $\calA_2$ to $U_2$ is Cohen--Macaulay since $U_2$ is
  obtained by removing the non-Cohen--Macaulay locus $Y_2$ of $Z$.
\end{proof}
Let $\injfunc j U Z$ be the open immersion. We will show that the
direct image $j_*\calA_U$ gives rise to a lift $C$ of the curve
$C_K$. To see this, we need the following two results.

\begin{prop}[{\cite[Corollaire (5.11.4)]{EGAIV2}}]\label{prop:coherent_EGAIV}
  Let $X$ be a locally Noetherian scheme that can locally be embedded
  into a regular scheme. Let $U$ be an open subset of $X$, and denote
  by $\injfunc i U X$ the inclusion. Let $\calF$ be a coherent
  $\calO_U$-module. The direct image $i_*\calF$ is coherent if and only if
  $\codim((X\setminus U) \cap \overline{\{x\}},\overline{\{x\}})\geq 2$
  for every $x\in\Ass(\calF)$.
\end{prop}
\begin{prop}[{\cite[Proposition
    (5.10.10)(i)]{EGAIV2}}]\label{prop:S2_EGAIV}
  Let $X$ be a locally Noetherian scheme, let $U$ be an open subscheme
  of $X$, and denote by $\injfunc i U X$ the inclusion. Let $\calF$ be
  a coherent $\calO_U$-module, and suppose that the direct image
  $i_*\calF$ is coherent. Then $\depth_{X\setminus U}(i_*\calF) \geq 2$.
\end{prop}

Applying these results to the coherent $\calO_U$-algebra $\calA_U$
that we defined above gives the following.
\begin{prop}\label{prop:dir_image_coherent}
  Let the notation be as above. The sheaf $\calA:=j_*\calA_U$ of
  $\calO_Z$-algebras is coherent and Cohen--Macaulay as an 
  $\calO_Z$-module.
\end{prop}
\begin{proof}
  We have to show that $\calA_U$ satisfies the properties of
  Proposition~\ref{prop:coherent_EGAIV}. By Lemma~\ref{lm:A_U_is_CM},
  the sheaf $\calA_U$ is Cohen--Macaulay, and hence the associated
  points of $\calA_U$ are the generic points of
  $\supp(\calA_U)=U$. Thus we have to show that $\codim((Z\setminus
  U)\cap Z',Z')=2$ for every irreducible component $Z'$ of $Z$. By
  Proposition~\ref{prop:codim_2}, the closed set $Z\setminus U$ has
  codimension $2$ in $Z$, and hence it consists of closed points. We
  have seen in Corollary~\ref{cor:codim_point} that every closed point
  has codimension $2$ in every irreducible component containing
  it. This shows that $\calA$ is coherent.  Since $\depth_{Z\setminus
    U}\calA=2$ by Proposition~\ref{prop:S2_EGAIV}, it follows with
  Lemma~\ref{lm:A_U_is_CM} that $\calA$ is Cohen--Macaulay.
\end{proof}
Now we are ready to prove that the moduli functor $\CM$ is
proper.
\begin{thm}[Valuative criterion for properness]\label{thm:valu-crit-prop}
  Let $R$ be a discrete valuation ring with field of fractions
  $K$. For every commutative diagram \[\xymatrix{\spec(K)\ar[r] \ar[d]
    & \CM \ar[d] \\ \spec(R) \ar[r] \ar@{-->}[ur] & \spec(\bbZ)}\]
  there exists a unique map $\spec(R)\dasharrow \CM$ making the
  diagram commute.  In other words, every pair $(C_K,i_K) \in
  \CM(\spec(K))$ is the generic fiber of a unique element $(C,i)\in
  \CM(\spec(R))$.
\end{thm}
\begin{proof}
  Let $(C_K,i_K)\in \CM(\spec(K))$. The morphism $i_K$ factors as
  $C_K\stackrel{h_K}{\to} Z_K\hookrightarrow \bbP^n_K$, where $Z_K$ is
  the scheme-theoretic image of $C_K$ in $\bbP^n_K$. Let $Z$ be the
  scheme-theoretic closure of $Z_K$ in $\bbP^n_R$. Let further $Y_1$
  be the closure of $Y_K:=\supp(\coker(h_K^\#))$ in $Z$ and $Y_2$ be
  the non-Cohen--Macaulay locus of $Z$. Let $Y:=Y_1\cup Y_2$, and let
  $U_1:=Z_K$, $U_2:=Z\setminus Y$ and $U:=U_1\cup U_2$. The algebras
  $(h_K)_*\calO_{C_K}$ on $U_1$ and $\calO_{U_2}$ on $U_2$ coincide on
  the intersection $U_1\cap U_2$, and glue to a coherent
  $\calO_U$-algebra $\calA_U$.  We set $\calA:=j_*(\calA_U)$, where
  $\injfunc j U Z$ denotes the inclusion. Let
  $C:=\mathbf{Spec}(\calA)$, and let $\func i C \bbP^n_R$ be the
  composition of the structure map $C\to Z$ and the embedding
  $Z\hookrightarrow \bbP^n_R$. Note that base change to $\spec (K)$
  gives the finite morphism $\func {i_K}{C_K}{\bbP^n_K}$.

  Finally, we claim that $(C,i)$ defines an element in
  $\CM(\spec(R))$. First we observe that, by
  Proposition~\ref{prop:dir_image_coherent}, the $\calO_Z$-module
  $\calA$ is coherent, and therefore the morphism $i$ is
  finite. Moreover, as we have $\Ass(\calA)=\Ass(\calA_U)$ by
  \cite[Proposition (3.1.13)]{EGAIV2}, all associated points of
  $\calA$ lie in the generic fiber over $\spec(R)$. By
  Lemma~\ref{lm:flat_over_DVR}, it follows that $C$ is flat over
  $\spec(R)$.  It remains to show that the pair $(C,i)$ satisfies the
  conditions on the fibers over $\spec(R)$. Note that restriction to
  the generic point gives the element $(C_K,i_K)$ that lies in
  $\CM(\spec(K))$ by assumption. Hence, we only have to check the
  closed fiber. It follows from Proposition~\ref{prop:dim_fibers} that
  also the closed fiber $C_0$ has pure dimension $1$. Moreover, the
  sheaf $\calA$ is Cohen--Macaulay as an $\calO_Z$-module by
  Proposition~\ref{prop:dir_image_coherent}, and hence also as an
  $\calA$-module, see
  Lemma~\ref{lm:CM_base_change}\ref{lm:CM_base_change1}. It follows
  that the scheme $C$ is Cohen--Macaulay. The closed fiber $C_0$ is
  obtained by dividing out with the non zero divisor $\pi$, and it is
  therefore Cohen--Macaulay by \cite[Theorem~
  2.1.3(a)]{Bruns-Herzog}. Finally, the closed fiber $Y_0$ of $Y$ has
  dimension $0$ by Proposition~\ref{prop:codim_2}. By
  Lemma~\ref{lm:iso_alt}, this implies that also the restriction
  $\func{i_0}{C_0}{\bbP^n_{k}}$ is an isomorphism onto its image away
  from finitely many closed points. This concludes the proof since the
  Hilbert polynomial is locally constant.
\end{proof}

\newcommand{\etalchar}[1]{$^{#1}$}
\providecommand{\bysame}{\leavevmode\hbox to3em{\hrulefill}\thinspace}
\providecommand{\MR}{\relax\ifhmode\unskip\space\fi MR }
\providecommand{\MRhref}[2]{%
  \href{http://www.ams.org/mathscinet-getitem?mr=#1}{#2}
}
\providecommand{\href}[2]{#2}


\begin{thebibliography}{FGI{\etalchar{+}}05}

\bibitem[AK10]{AK:branched}
Valery Alexeev and Allen Knutson, \emph{Complete moduli spaces of
  branchvarieties}, J. Reine Angew. Math. \textbf{639} (2010), 39--71.

\bibitem[Art69]{Artin}
Michael Artin, \emph{Algebraization of formal moduli. {I}}, Global {A}nalysis
  ({P}apers in {H}onor of {K}. {K}odaira), Univ. Tokyo Press, Tokyo, 1969,
  pp.~21--71.

\bibitem[BH93]{Bruns-Herzog}
Winfried Bruns and J{\"u}rgen Herzog, \emph{Cohen-{M}acaulay rings}, Cambridge
  Studies in Advanced Mathematics, vol.~39, Cambridge University Press,
  Cambridge, 1993.

\bibitem[Bou72]{Bourbaki:CA}
Nicolas Bourbaki, \emph{Elements of mathematics. {C}ommutative algebra.
  {E}nglish translation.}, Hermann, Paris, 1972.

\bibitem[Eis95]{Eisenbud}
David Eisenbud, \emph{Commutative algebra. {W}ith a view toward algebraic
  geometry}, Graduate Texts in Mathematics, vol. 150, Springer-Verlag, New
  York, 1995.

\bibitem[FGI{\etalchar{+}}05]{FGAexplained}
Barbara Fantechi, Lothar G{\"o}ttsche, Luc Illusie, Steven~L. Kleiman, Nitin
  Nitsure, and Angelo Vistoli, \emph{Fundamental algebraic geometry.
  {G}rothendieck's {FGA} explained}, Mathematical Surveys and Monographs, vol.
  123, American Mathematical Society, Providence, RI, 2005.

\bibitem[Gro61a]{EGAII}
Alexander Grothendieck, \emph{\'{E}l\'ements de g\'eom\'etrie alg\'ebrique
  {II}. \'{E}tude globale \'el\'ementaire de quelques classes de morphismes},
  Inst. Hautes \'Etudes Sci. Publ. Math. \textbf{8} (1961), 1--222.

\bibitem[Gro61b]{EGAIII1}
\bysame, \emph{\'{E}l\'ements de g\'eom\'etrie alg\'ebrique {III}. \'{E}tude
  cohomologique des faisceaux coh\'erents. {1}}, Inst. Hautes \'Etudes Sci.
  Publ. Math. \textbf{11} (1961), 1--167.

\bibitem[Gro61c]{FGA221}
\bysame, \emph{Techniques de construction et th\'eor\`emes d'existence en
  g\'eom\'etrie alg\'ebrique. {IV}. {L}es sch\'emas de {H}ilbert}, S\'eminaire
  {B}ourbaki, {V}ol.\ 6, no. 221, Soc. Math. France, 1961, pp.~249--276.

\bibitem[Gro63]{EGAIII2}
\bysame, \emph{\'{E}l\'ements de g\'eom\'etrie alg\'ebrique {III}. \'{E}tude
  cohomologique des faisceaux coh\'erents. {2}}, Inst. Hautes \'Etudes Sci.
  Publ. Math. \textbf{17} (1963), 1--91.

\bibitem[Gro65]{EGAIV2}
\bysame, \emph{\'{E}l\'ements de g\'eom\'etrie alg\'ebrique {IV}. \'{E}tude
  locale des sch\'emas et des morphismes de sch\'emas. {2}}, Inst. Hautes
  \'Etudes Sci. Publ. Math. \textbf{24} (1965), 1--231.

\bibitem[Gro66]{EGAIV3}
\bysame, \emph{\'{E}l\'ements de g\'eom\'etrie alg\'ebrique {IV}. \'{E}tude
  locale des sch\'emas et des morphismes de sch\'emas. {3}}, Inst. Hautes
  \'Etudes Sci. Publ. Math. \textbf{28} (1966), 1--255.

\bibitem[Gro67]{EGAIV4}
\bysame, \emph{\'{E}l\'ements de g\'eom\'etrie alg\'ebrique {IV}. \'{E}tude
  locale des sch\'emas et des morphismes de sch\'emas {4}}, Inst. Hautes
  \'Etudes Sci. Publ. Math. \textbf{32} (1967), 1--361.

\bibitem[GW10]{GortzWed:AlgGeo}
Ulrich G{\"o}rtz and Torsten Wedhorn, \emph{Algebraic geometry {I}}, Advanced
  Lectures in Mathematics, Vieweg + Teubner, Wiesbaden, 2010.

\bibitem[Har77]{Hartshorne}
Robin Hartshorne, \emph{Algebraic {G}eometry}, Graduate Texts in Mathematics,
  no.~52, Springer, New York, 1977.

\bibitem[Har94]{Hartshorne:genus}
\bysame, \emph{The {G}enus of {S}pace {C}urves}, Annali dell'Universita di
  Ferrara \textbf{40} (1994), 207--223.

\bibitem[H{\o}n05]{Honsen}
Morten H{\o}nsen, \emph{Compactifying locally {C}ohen-{Ma}caulay projective
  curves}, Ph{D} thesis, Kungliga Tekniska h{\"o}gskolan Stockholm, 2005.

\bibitem[Kle71]{Kleiman:m-reg}
Steven~L. Kleiman, \emph{Les {T}h\'eor\`emes de {F}initude pour le {F}oncteur
  de {P}icard}, Th\'eorie des {I}ntersections et {T}h\'eor\`eme de
  {R}iemann-{R}och, Lecture Notes in Mathematics, vol. 225, Springer
  Berlin/Heidelberg, 1971, pp.~616--666.

\bibitem[Kol09]{Kollar:HH}
J\'anos Koll\'ar, \emph{{H}ulls and {H}usks}, 2009, Preprint, arXiv:0805.0576v4
  [math.AG].

\bibitem[LMB00]{LMB:Champs}
G{\'e}rard Laumon and Laurent Moret-Bailly, \emph{Champs alg\'ebriques},
  Ergebnisse der Mathematik und ihrer Grenzgebiete. 3. Folge. A Series of
  Modern Surveys in Mathematics [Results in Mathematics and Related Areas. 3rd
  Series. A Series of Modern Surveys in Mathematics], vol.~39, Springer-Verlag,
  Berlin, 2000.

\bibitem[Mat80]{Mats:algebra}
Hideyuki Matsumura, \emph{Commutative algebra}, W. A. Benjamin, Inc., New York,
  1980.

\bibitem[Mil80]{Milne:EC}
James~S. Milne, \emph{\'{E}tale cohomology}, Princeton Mathematical Series,
  vol.~33, Princeton University Press, Princeton, N.J., 1980.

\bibitem[Mum66]{Mumford:Lectures}
David Mumford, \emph{Lectures on curves on an algebraic surface}, With a
  section by G. M. Bergman. Annals of Mathematics Studies, No. 59, Princeton
  University Press, Princeton, NJ, 1966.

\bibitem[PT09]{RT:stable_pairs}
Rahul Pandharipande and Richard Thomas, \emph{Curve counting via stable pairs
  in the derived category}, Inventiones Mathematicae \textbf{178} (2009),
  407--447.

\bibitem[Ryd08]{Rydh:thesis}
David Rydh, \emph{Families of cycles and the {C}how scheme}, Ph{D} thesis,
  Kungliga Tekniska h{\"o}gskolan Stockholm, Mathematics (Dept.), 2008.

\end{thebibliography}
 \end{document}